\pdfoutput=1
\documentclass[11pt]{article}
\usepackage{amssymb}
\usepackage[all,cmtip]{xy}
\usepackage[width=17cm,top = 3.5cm,bottom=3.5cm]{geometry}
\usepackage{tikz}
\usepackage[all]{xy}
\usetikzlibrary{matrix,arrows,decorations.pathmorphing}
\usepackage{amsthm,amsfonts,amssymb,amsmath,amscd} 
\usepackage{aliascnt} 
\usepackage{mathrsfs} 
\usepackage{xargs,ifthen} 

\usepackage{comment}
\usepackage{color}

\usepackage{hyperref}
\definecolor{tocolor}{rgb}{.1,.1,.1}
\definecolor{urlcolor}{rgb}{.2,.2,.6}
\definecolor{linkcolor}{rgb}{.1,.1,.5}
\definecolor{citecolor}{rgb}{.4,.2,.1}
\hypersetup{backref=true, colorlinks=true, urlcolor=urlcolor, linkcolor=linkcolor, citecolor=citecolor}

\newcommandx{\thdef}[2]{
	\newaliascnt{#1}{theorem}  
	\newtheorem{#1}[#1]{#2}
	\aliascntresetthe{#1}  
	\newtheorem*{#1*}{#2}
	\expandafter\newcommand\expandafter{\csname #1autorefname\endcsname}{#2}
}

\makeatletter
\newtheorem*{rep@theorem}{\rep@title}
\newcommand{\newreptheorem}[2]{%
\newenvironment{rep#1}[1]{%
 \def\rep@title{#2 \ref{##1}}%
 \begin{rep@theorem}}%
 {\end{rep@theorem}}}
\makeatother

\newtheorem{theorem}{Theorem}[section]
\newreptheorem{theorem}{Theorem}

\thdef{lemma}{Lemma}
\thdef{corollary}{Corollary}
\thdef{conjecture}{Conjecture}
\newreptheorem{conjecture}{Conjecture}
\thdef{proposition}{Proposition}

\theoremstyle{definition}
\thdef{definition}{Definition}
\thdef{notation}{Notation}

\theoremstyle{remark}
\thdef{remark}{Remark}

\theoremstyle{remark}
\thdef{ex}{Example}

\newenvironment{example}
{\begin{ex}}%
{\hfill $\blacksquare$\end{ex}}

\newcommand{\spc}[1]{\mathsf{#1}} 
\newcommand{\shf}[1]{\mathcal{#1}} 

\newcommand{\rbrac}[1]{\left(#1\right)} 

\newcommandx{\fn}[2][2=]{#1\ifthenelse{\equal{#2}{}}{}{\!\rbrac{{#2}}}} 
\newcommandx{\id}[2][2=]{\fn{{\rm id}_{#1}}[#2]} 

\newcommand{\ext}[2][\bullet]{\spc{\Lambda}^{#1}{#2}} 
\newcommandx{\End}[2][1=]{\fn{\spc{End}_{#1}}[#2]} 
\newcommandx{\Hom}[2][1=]{\fn{\spc{Hom}_{#1}}[#2]} 
\newcommandx{\Aut}[2][1=]{\fn{\spc{Aut}_{#1}}[#2]} 
\newcommandx{\image}[1]{\fn{\spc{img}}[#1]} 
\renewcommandx{\ker}[1]{\fn{\spc{ker}}[#1]} 
\newcommandx{\rank}[1]{\fn{\mathrm{rank}}[#1]} 

\newcommandx{\ann}[1]{\fn{\spc{ann}}[#1]} 

\newcommandx{\hlgy}[3][1=\bullet,3=]{\spc{H}_{#1}^{#3}\!\rbrac{{#2}}} 
\newcommandx{\cohlgy}[3][1=\bullet,3=]{\spc{H}^{#1}_{#3}\!\rbrac{{#2}}} 
\newcommandx{\chow}[3][1=\bullet,3=]{\spc{A}^{#1}_{#3}\!\rbrac{{#2}}} 

\newcommandx{\Ext}[3][1=\bullet,3=]{\fn{\spc{Ext}^{#1}_{#3}}[{#2}]} 
\newcommandx{\Tor}[3][1=\bullet,3=]{\fn{\spc{Tor}^{#1}_{#3}}[{#2}]} 

\newcommandx{\Pic}[1]{\fn{\spc{Pic}}[{#1}]} 
\newcommandx{\chernalg}[2][1=\bullet]{\fn{\spc{Chern}^{#1}}[{#2}]} 
\newcommandx{\chern}[2][1=]{\fn{c_{#1}}[#2]} 
\newcommandx{\ch}[2][1=]{\fn{\mathrm{ch}_{#1}}[{#2}]} 

\newcommandx{\sKer}[2][1=]{ \fn{ \shf{K}er_{#1}}[{#2}] } 
\newcommandx{\sHom}[2][1=]{ \fn{ \shf{H}om_{#1}}[{#2}] } 
\newcommandx{\sEnd}[2][1=]{ \fn{ \shf{E}nd_{#1}}[{#2}] } 
\newcommandx{\sExt}[3][1=\bullet,3=]{\fn{\shf{E}xt^{#1}_{#3}}[{#2}]} 
\newcommandx{\sTor}[3][1=\bullet,3=]{\fn{\shf{T}or^{#1}_{#3}}[{#2}]} 

\newcommandx{\forms}[2][1=\bullet]{\Omega^{#1}_{#2}} 
\newcommandx{\can}[1][1=]{\omega_{#1}} 
\newcommandx{\acan}[1][1=]{\omega_{#1}^{-1}} 
\newcommandx{\tshf}[1]{\shf{T}_{#1}} 
\newcommandx{\mvect}[2][1=\bullet]{ \ext[#1]{\tshf{#2}} }
\newcommandx{\der}[2][1=\bullet]{\mathscr{X}^{#1}_{#2}} 
\newcommandx{\sJet}[3][1=,2=]{\shf{J}^{#1}_{#2}#3} 

\newcommandx{\tb}[2][1=]{\spc{T}_{\!#1}{#2}} 
\newcommandx{\ctb}[2][1=]{\spc{T}_{\!#1}^*{#2}} 
\newcommandx{\lie}[2][2=]{\fn{\mathscr{L}_{#1}}[#2]} 
\newcommandx{\hook}[2][2=]{\fn{i_{#1}}[#2]} 

\newcommand{\Bl}[2]{\fn{{\mathrm{Bl}_{#2}}}[{#1}]}

\newcommand{\Rep}[1]{\fn{\cat{R}ep}[#1]}

\makeatletter
\newcommand{\thickbar}{\mathpalette\@thickbar}
\newcommand{\@thickbar}[2]{{#1\mkern1.5mu\vbox{
  \sbox\z@{$#1\mkern-1mu#2\mkern-1mu$}%
  \sbox\tw@{$#1\overline{#2}$}%
  \dimen@=\dimexpr\ht\tw@-\ht\z@-.6\p@\relax
  \hrule\@height.4\p@ 
  \vskip1\p@
  \hrule\@height.4\p@ 
  \vskip\dimen@
  \box\z@}\mkern1.5mu}
}
\makeatother

\newcommand{\Res}{\mathrm{Res}}

\def\Rep{\text{Rep}}

\def\ker{\text{ker}}
\def\im{\text{im}}
\def\End{\text{End}}
\def\log{\text{log}}

\newcommand{\bb}[1]{\mathbb{#1}}
\renewcommand{\cal}[1]{\mathcal{#1}}

\numberwithin{equation}{section}
\newtheoremstyle{parag}
  {\topsep}   
  {\topsep}   
  {}  
  {}       
  {\bfseries} 
  {.}         
  { } 
  {}          
\theoremstyle{parag}

\makeatletter
\def\@cite#1#2{{\normalfont[{#1\if@tempswa , #2\fi}]}}
\makeatother

\renewcommand{\Pic}{\mathrm{Pic}}

\def\rep{\text{Rep}}
\def\Aut{\text{Aut}}
\def\Bl{\text{Bl}}

\begin{document}

\title{\vspace{-4em} \huge Lie groupoids and logarithmic connections}
\date{}

\author{
Francis Bischoff\thanks{Exeter College and Mathematical Institute, University of Oxford; {\tt francis.bischoff@maths.ox.ac.uk }}
}
\maketitle

\abstract{
Using tools from the theory of Lie groupoids, we study the category of logarithmic flat connections on principal $G$-bundles, where $G$ is a complex reductive structure group. Flat connections on the affine line with a logarithmic singularity at the origin are equivalent to representations of a groupoid associated to the exponentiated action of $\mathbb{C}$. We show that such representations admit a canonical Jordan-Chevalley decomposition and use this to give a functorial classification. Flat connections on a complex manifold with logarithmic singularities along a hypersurface are equivalent to representations of a twisted fundamental groupoid. Using a Morita equivalence, whose construction is inspired by Deligne's notion of paths with tangential basepoints, we prove a van Kampen type theorem for this groupoid. This allows us to show that the category of representations of the twisted fundamental groupoid can be localized to the normal bundle of the hypersurface. As a result, we obtain a functorial Riemann-Hilbert correspondence for logarithmic connections in terms of generalized monodromy data. 
 
}
\renewcommand{\contentsname}{}
\setcounter{tocdepth}{1}
\tableofcontents

\section*{Introduction}
This paper is about the classification of flat connections on principal $G$-bundles with logarithmic singularities, where $G$ is a connected complex reductive group. These connections are the differential geometric generalizations of linear ordinary differential equations with Fuchsian singularities. In their most simple incarnation, these are differential equations of the form 
\[
z \frac{ds}{dz} = A(z) s(z).
\]
In this equation, $A(z)$ is a matrix of holomorphic functions on the affine line, or more generally, a family of holomorphically varying elements of a Lie algebra. Two such equations are considered to be equivalent when they are related by a holomorphic gauge transformation. There has been much work in the past devoted to establishing standard normal forms for these differential equations and in describing their classification under holomorphic gauge transformations \cite{hukuhara1937proprietes, alllevelt1961hypergeometric, turrittin1955convergent, gantmacher1959theory, babbitt1983formal, kleptsyn2004analytic, boalch2011riemann}. In the setting of complex vector bundles with logarithmic connections, there are classification results due to Deligne \cite[Appendix C]{esnault1986logarithmic} and Simpson \cite{simpson1990harmonic}, as well as Ogus \cite{ogus2003logarithmic} in the context of logarithmic geometry. If we allow the gauge transformations to be meromorphic at the singularity locus, then we are in the setting of connections with regular singularities, and these have been much studied. Of particular importance is the work of Deligne \cite{deligne2006equations} (see also \cite{babbitt1983formal}). In this paper, we extend existing results by giving a completely functorial classification of logarithmic connections on principal $G$-bundles over arbitrary complex manifolds in terms of generalized monodromy data. Part of the novelty of our approach is that it draws heavily from the theory of Lie groupoids. This allows us to give new proofs of known results which completely circumvent the use of power series, and which do not rely on results in analysis beyond the classic existence and uniqueness theorem for first order ordinary differential equations.

The starting point for our approach is the perspective, explained in \cite{gualtieri2018stokes}, that \emph{singular} flat connections can be fruitfully studied as \emph{smooth} flat algebroid connections for an appropriately chosen Lie algebroid. Such a Lie algebroid may be integrated to a Lie groupoid, which is a space on which the sections of the Lie algebroid are realized as vector fields tangent to a foliation. If we then pull back a singular connection to the groupoid it becomes a smooth flat connection defined along the foliation. It may therefore be integrated using the basic existence and uniqueness results for ordinary differential equations. The flat sections thus obtained are smooth and single-valued, and give rise to a representation of the Lie groupoid. This representation may in turn be used to recover the flat sections of the original singular connection, but these sections are typically singular and multivalued. In this way, the groupoid provides the natural domain of definition for the solutions to singular flat connections. From this vantage point, the fact that the solutions to singular differential equations are singular and multivalued is simply an indication that we are working on the wrong space. 

The passage from flat Lie algebroid connections to Lie groupoid representations, a special case of Lie's second theorem for Lie groupoids \cite{mackenzie2000integration, moerdijk2002integrability}, gives an equivalence of categories. This can be seen as the first step in establishing a Riemann-Hilbert correspondence theorem. Indeed, in the case of the tangent Lie algebroid to a manifold $M$, we get the equivalence between smooth flat connections and representations of the fundamental groupoid $\Pi(M)$, obtained by taking the parallel transport. The classical Riemann-Hilbert correspondence then follows from the observation that $\Pi(M)$ and $\pi_{1}(M,x)$, the fundamental group based at the point $x \in M$, are Morita equivalent, and so have equivalent categories of representations. In this paper, we are interested in flat connections on a complex manifold $X$ with logarithmic singularities along a smooth hypersurface $D$. The relevant Lie algebroid in this case is the log tangent algebroid $T_{X}(- \log D)$, which integrates to the twisted fundamental groupoid $\Pi(X,D)$. Our study of logarithmic connections therefore comes down to a study of the representations of this groupoid. 

In Section \ref{Firstsection} of the paper, we study flat connections on the affine line which have a logarithmic singularity at the origin. The twisted fundamental groupoid, in this case, is simply given by an action groupoid $\mathbb{C} \ltimes \mathbb{A}$. In Theorem \ref{RH}, we construct an explicit pair of inverse functors between the category of $\mathbb{C} \ltimes \mathbb{A}$ representations and a category of generalized monodromy data $F((\bb{C}, 0), G)$, which we describe below. The objects of this category are tuples $(M, K_{1}, \nu_{0}, K_{0}, A)$, where $K_{1}$ and $K_{0}$ are right $G$-torsors, $M \in Aut_{G}(K_{1})$ is a monodromy element, $A \in \frak{aut}_{G}(K_{0})$ is a residue element, and $\nu_{0} \subset \spc{Hom}_{G}(K_{0},K_{1})$ is a torsor for a unipotent group which is determined by $A$ and which is interpreted as a space of regularized parallel transport maps. These data are furthermore required to satisfy a compatibility condition, which is described in Section \ref{Definingthecat}. We think of the objects in this category as in the following diagram. 
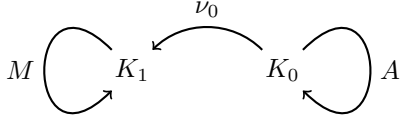
\begin{center}
\begin{tikzpicture}
\node (A) at (-1,0) {$K_{1}$};
\node (B) at (1,0) {$K_{0}$};

\draw[->, thick] (B) to [out = 135, in=45] node [above] {$\nu_{0}$} (A);
\path (B) edge [->, thick, out=45,in=-45,looseness=8] node[right] {$A$} (B);
\path (A) edge [->, thick, out=135,in=225,looseness=8] node[left] {$M$} (A);
\end{tikzpicture}
\end{center}
If we decategorify Theorem \ref{RH}, then we recover the description of the isomorphism classes of ordinary differential equations with Fuschian singularities established in \cite{babbitt1983formal, kleptsyn2004analytic}. Furthermore, we can substitute $F((\bb{C},0), G)$ with an equivalent category $E((\bb{C},0), G)$, where $\nu_{0}$ is replaced by a torsor for the unipotent radical of a parabolic subgroup of $Aut_{G}(K_{0})$. Looking at the isomorphism classes of objects in this category then leads to the classification result of \cite[Theorem A]{boalch2011riemann}. 

We will now briefly outline the proof of Theorem \ref{RH}, and note the differences with some of the existing approaches. A representation $\Phi$ of $\bb{C} \ltimes \bb{A}$ has a canonical linear approximation $L(\Phi)$, which corresponds to a trivial connection of the form 
\[
\nabla = d - R\frac{dz}{z},
\]
for constant $R$. Since such trivial connections are fully characterized by their residues, it is tempting to approach the classification problem by searching for a \emph{linearization} isomorphism between $\Phi$ and $L(\Phi)$. However, the phenomenon of resonance obstructs certain representations from admitting linearizations. There are several existing methods for getting around this problem, which involve abandoning linearizations in favor of less canonical normal forms. When the structure group is $GL(n, \bb{C})$, it is possible to use a meromorphic gauge transformation to trivialize the connection so that the connection $1$-form is constant and satisfies a normalization condition on its eigenvalues. This leads to a classification of logarithmic connections given in terms of the data of a vector space $V$ equipped with a filtration $F$ and a compatible monodromy matrix $M$. The filtration $F$ was first described by Levelt \cite{levelt1961hypergeometric} and it encodes the asymptotic behavior of the meromorphic gauge transformation. This approach to classification fails to generalize to the case of an arbitrary reductive structure group. As explained in \cite{babbitt1983formal, boalch2011riemann}, this is because the exponential map may not be surjective. Note also that the trivial connection arising in this approach is in general not isomorphic to the linear approximation of the original connection. 

Another approach to classification is to deal directly with the more complicated normal forms which exist for logarithmic connections. Using a holomorphic gauge transformation, it is possible to trivialize a logarithmic connection so that it is in the Levelt normal form
\[
A(z) = S + \sum_{i \geq 0} z^{i}N_{i},
\]
where $S$ is semisimple, and $N_{i}$ are nilpotent and satisfy $[S,N_{i}] = iN_{i}$. The Levelt normal form of a given connection is not unique, and the classification results of \cite{babbitt1983formal, kleptsyn2004analytic} are obtained by analyzing the holomorphic equivalence of different Levelt normal forms. The approach of \cite{boalch2011riemann} also works with the Levelt normal form, but gives a classification in terms of a monodromy element $M$ contained in a parabolic subgroup $P$, generalizing the classification using the Levelt filtration. 

In this paper, we do not abandon linearizations, but rather circumvent the problem posed by resonance by using a functorial Jordan-Chevalley decomposition for representations. Namely, in Theorem \ref{JCdecomp} we show that a given representation $\Phi$ factors into a representation with semisimple monodromy $\Phi_{s}$ and a unipotent automorphism $U$. This overcomes the difficulty posed by resonance because representations with semisimple monodromy always admit linearizations. Indeed, by using an argument based on linearizing $\mathbb{C}^{*}$ actions, we show in Proposition \ref{sslinearize} that a representation is linearizable if and only if its monodromy is conjugate to the exponential of its residue. As a result, we obtain Theorem \ref{RH}, in which a representation $\Phi$ is classified by the data of its monodromy $M$, its residue $A$, and the space $\nu_{0}$ of strict linearizations of its semisimple component $\Phi_{s}$. This provides a much cleaner approach to classification, and because it does not rely on arbitrary choices, it is completely functorial. As a corollary, we also obtain a novel proof of the Levelt normal form, and consequently of the fact that logarithmic singularities are regular. 

In Section \ref{Globalsection} of the paper, we turn to the general case of flat connections on a complex manifold $X$ with logarithmic singularities along a smooth hypersurface $D$. The main tool here is Morita equivalence, which allows us to decompose the twisted fundamental groupoid $\Pi(X,D)$ into simple building blocks. These building blocks consist of the fundamental group of the complement $X \setminus D$, and certain action groupoids defined on the fibers of the normal bundle to $D$. In this way, the problem of classifying logarithmic connections is largely reduced to the $1$-dimensional case, and hence we can readily apply the results from Section \ref{Firstsection}. 

In Theorem \ref{FinalRH}, we prove an equivalence between the category of flat connections on $X$ with logarithmic singularities along $D$, and a category $F( (X,D), G)$ of generalized monodromy data. The category $F( (X,D), G)$ depends on the choice of a base point $x_{0}$ in $X \setminus D$ and a collection of paths $p_{i}$ going from $x_{0}$ to the connected components of $D$. The details of the definition are given in Section  \ref{RHsection}. Very roughly, an object in this category consists of the data of a residue element for each component of $D$, a space of regularized parallel transport maps for each path $p_{i}$, and a representation of $\pi_{1}(X \setminus D, x_{0})$. We stress here that the equivalence of Theorem \ref{FinalRH} is given by a pair of inverse functors, although we do not write them out explicitly. 

The groupoids constructed in Section \ref{Globalsection}, and the Morita equivalences relating them, are interesting in their own right and may be interpreted in terms of Deligne's idea of paths with tangential basepoints \cite{deligne1989groupe}. Furthermore, the results on Morita equivalence may be understood both as an instance of Zung's linearization theorem \cite{zung2006proper, weinstein2002linearization, crainic2013linearization} appropriately generalized to Lie groupoids in the holomorphic category, and as a generalization of the van Kampen theorem with several basepoints \cite{brown1967groupoids, brown1984van}. 

\vspace{.05in}

\noindent \textbf{Acknowledgements.} 
I began this project at the start of my Ph.D. studies in 2015 under the supervision of M. Gualtieri, and several ideas were worked out during a visit to Paris in 2016 where I had the opportunity to talk with P. Boalch. In the end, it took the confinement imposed by the coronavirus pandemic to force me to finalize the details. Many of the ideas of this project were developed in collaboration with M. Gualtieri, in particular the work in the final section. I would like to thank M. Gualtieri, P. Boalch, and B. Pym for many useful discussions. I am supported by an NSERC postdoctoral fellowship.

\section{Lie groupoids and differential equations}
In this section, we explain how to recast the notion of flat logarithmic connections into the language of Lie groupoids. For an introduction to Lie groupoids and Lie algebroids, we recommend \cite{mackenzie2005general, crainic2011lectures, MR2012261}. 

Let $G$ be a complex Lie group, with Lie algebra $\frak{g}$. The Atiyah algebroid of a principal $G$-bundle $p : P \to M$ is the Lie algebroid $At(P) = TP/G$ of $G$-invariant vector fields on $P$. It is a Lie algebroid over $M$ and it sits in the following short exact sequence 
\begin{equation} \label{Atiyah sequence}
0 \to Ad(P) \to At(P) \to TM \to 0,
\end{equation}
where $Ad(P) = P \times_{G} \frak{g}$ is the adjoint bundle, and the projection $dp: At(P) \to TM$ is the anchor map. Trivialising the bundle $P \cong M \times G$ gives rise to a splitting of sequence \ref{Atiyah sequence}. This gives a concrete description of the Atiyah algebroid as $At(P) \cong TM \oplus \frak{g}$, with bracket
\[
[(V, a), (U, b)] = ([V, U], V(b) - U(a) + [a,b]),
\]
where $V, U$ are vector fields, and $a, b$ are sections of $\frak{g} \times M$. 

One of the Lie groupoids integrating the Atiyah algebroid is the gauge groupoid $\cal{G}(P)$. This is the groupoid over $M$ whose space of arrows is given by $(P \times P)/G$, interpreted as the space of $G$-equivariant maps between the fibers of $P$. A trivialization of $P$ induces an isomorphism $\cal{G}(P) \cong M \times G \times M$, for which the target and source maps are given, respectively, by 
\[
t(x, g, y) = x, \qquad s(x, g, y) = y,
\]
and the groupoid multiplication is given by 
\[
(x, g, y) \ast (y, h, z) = (x, gh, z). 
\]

A $G$-representation of a Lie algebroid $A \to M$ is a Lie algebroid homomorphism $\nabla : A \to At(P)$ into the Atiyah algebroid of a principal $G$-bundle $P$. These form a category, which we denote $\text{Rep}(A, G)$. For example, when the Lie algebroid is the tangent bundle $TM$, representations coincide with smooth flat connections over $M$. 

Let $X$ be a complex manifold with a smooth hypersurface $D$. The logarithmic tangent bundle $T_{X}(- \log D)$ is the Lie algebroid whose local sections are vector fields on $X$ which are tangent to $D$. In the case of the affine line $\bb{A}$ with $D$ given by the origin, the logarithmic tangent bundle $T_{\bb{A}}(-\log \  0)$ has a global non-vanishing section, whose image under the anchor map is the Euler vector field $z \partial_{z}$. 

A flat connection on $X$ with logarithmic singularities along $D$ is defined to be a representation of $T_{X}(- \log D)$. More precisely, this is a principal bundle $P$ and a homomorphism of Lie algebroids
\[
\nabla: T_{X}(- \log D) \to At(P). 
\]
This agrees with the standard definition, as given in \cite[Chapter 2]{deligne2006equations}. Logarithmic flat connections are a generalization of linear ordinary differential equations with Fuchsian singularities. Indeed, a logarithmic connection on the trivial bundle over the affine line is given by a morphism ${\nabla : T_{\bb{A}}(- \log \ 0) \to T\bb{A} \oplus \frak{g}}$, which is specified by its value on a global section as follows
\[
\nabla(z \partial_{z}) = (z \partial_{z}, -A(z)). 
\]
This corresponds to the following ordinary differential equation
\begin{equation} \label{LogODE}
z \frac{ds}{dz} = A(z)s(z),
\end{equation}
which has a Fuchsian singularity at the origin. In this equation, $s$ is a $G$-valued function, which is a solution to the differential equation and a flat section of the logarithmic connection. Changing this equation using a gauge transformation is equivalent to changing the trivialization of the principal bundle. 

A $G$-representation of a Lie groupoid $\cal{G} \rightrightarrows M$ is a homomorphism $\Phi : \cal{G} \to \cal{G}(P)$ into the gauge groupoid of a principal $G$-bundle $P$. These form a category, which we denote $\text{Rep}(\cal{G}, G)$. Since the gauge groupoid of the trivial bundle is isomorphic to $M \times G \times M$, a representation of $\cal{G}$ on the trivial bundle is determined by a homomorphism $\phi: \cal{G} \to G$. 

Let $A$ be the Lie algebroid of $\cal{G}$. By differentiating a representation $\Phi$ of $\cal{G}$, we obtain a representation of $A$
\[
d\Phi : A \to At(P).
\]
Conversely, if $\cal{G}$ has connected and simply connected source fibers, then Lie's second theorem \cite{mackenzie2000integration, moerdijk2002integrability} implies that a representation of $A$ integrates in a unique way to a representation of $\cal{G}$. It is worth noting that the analysis required for constructing this integration is limited to the classical existence and uniqueness result for first order ordinary differential equations (see \cite[Theorem 3.8]{gualtieri2018stokes}). This defines an equivalence between the categories of representations 
\[
\text{Rep}(\cal{G}, G) \cong \text{Rep}(A, G).
\]

For example, the source simply connected Lie groupoid that integrates the tangent bundle $TM$ of a manifold is the fundamental groupoid $\Pi(M)$, consisting of homotopy classes of paths in $M$. Hence, in this case, we recover the well-known fact that a flat connection on a principal bundle $P$ is equivalent to a representation of the fundamental groupoid 
\[
\Phi : \Pi(M) \to \cal{G}(P). 
\]
Indeed, the $G$-equivariant map $\Phi(\gamma) : P_{\gamma(0)} \to P_{\gamma(1)}$ associated to a path $\gamma$ in $M$ is given by the parallel transport of the connection along this path. 

The logarithmic tangent bundle $T_{X}(- \log D)$ of a pair $(X,D)$ admits an integrating Lie groupoid by results of \cite{crainic2003integrability, debord2001holonomy}. The unique integration with connected and simply connected source fibers is the \emph{twisted fundamental groupoid} and is denoted $\Pi(X,D)$. In this case, Lie's second theorem gives an equivalence between the category of flat connections on $X$ with logarithmic singularities along $D$, and the category of representations of $\Pi(X,D)$
\[
\text{Rep}(T_{X}(- \log D), G) \cong \text{Rep}(\Pi(X,D), G).
\]
The upshot of this result is that the study of logarithmic connections is reduced to the representation theory of the twisted fundamental groupoid. 

The twisted fundamental groupoid $\Pi(\bb{A}, 0)$ for the pair consisting of the affine line and the origin is given by the action groupoid $\bb{C} \ltimes \bb{A}$, associated to the multiplicative action of $\bb{C}$ on $\bb{A}$. The space of arrows is the product $\bb{C} \times \bb{A}$, the target and source maps are defined, respectively, to be 
\[
t(\lambda, z) = e^{\lambda}z, \qquad s(\lambda, z) = z,
\]
and the multiplication is given by 
\[
(\lambda_{1}, e^{\lambda_{2}} z) \ast (\lambda_{2}, z) = (\lambda_{1} + \lambda_{2}, z). 
\]

For a general pair $(X,D)$, the space of arrows of the twisted fundamental groupoid may fail to be Hausdorff. This happens when the fundamental group of a tubular neighborhood of $D$ in $X \setminus D$ does not embed into the fundamental group of $X \setminus D$. By the results of \cite{del2020hausdorff}, there is a maximal Hausdorff Lie groupoid integrating $T_{X}(-\log D)$, and any representation factors through this groupoid. This means that we could choose to study representations using only Hausdorff Lie groupoids. However, since this does not lead to any simplifications, in this paper we focus solely on representations of $\Pi(X,D)$. 

\section{Refresher on reductive groups}
In this section, we review a portion of the theory of connected complex reductive groups, mainly following the reference \cite{humphreys2012linear} (see also \cite{borel2012linear}). 

An endomorphism $M$ of a vector space $V$ is \emph{semisimple} if it can be diagonalized, \emph{nilpotent} if $M^{k} = 0$ for some positive integer $k$, and \emph{unipotent} if $M - 1$ is nilpotent. These concepts carry over to the elements of a connected complex reductive group $G$ and its Lie algebra $\frak{g}$, and can be defined via faithful representations of $G$. An element $g \in G$ admits a unique multiplicative Jordan-Chevalley (JC) decomposition $g = su$, where $s$ is semisimple, $u$ is unipotent and $su = us$. An element $x \in \frak{g}$ admits a unique additive JC decomposition $x = s + n$, where $s$ is semisimple, $n$ is nilpotent and $[s,n] = 0$. The JC decomposition is preserved by morphisms of reductive groups, the exponential map $\exp : \frak{g} \to G$ sends the additive JC decomposition to the multiplicative JC decomposition, and the exponential map defines a bijection between the set of nilpotent elements of $\frak{g}$ and the set of unipotent elements of $G$. 

There is a decomposition of semisimple elements of $\frak{g}$ into real and imaginary parts, corresponding to the decomposition of their eigenvalues in representations that are induced by representations of $G$. Combining this with the additive JC decomposition, an arbitrary element $x \in \frak{g}$ admits a unique decomposition $x = a + ib + n$, where $a$ and $b$ are real semisimple, $n$ is nilpotent, and all mutually commute.

Let $T \subset G$ be a maximal torus, with Lie algebra $\frak{t}$. Its adjoint action on $\frak{g}$ induces the root space decomposition 
\[
\frak{g} = \frak{t} \oplus \bigoplus_{\alpha \in \Phi} \frak{g}_{\alpha},
\]
where $\Phi \subset \frak{t}^{*}$ is the set of roots, and 
\[
\frak{g}_{\alpha} = \{ v \in \frak{g} \ | \ Ad(t)(v) = \alpha(t)v, \text{ for all }t \in T \}.
\]
Furthermore, $\frak{g}_{\alpha}$ is a $1$-dimensional Lie algebra which integrates to a $1$-dimensional unipotent subgroup $U_{\alpha} \subset G$. Choosing a Borel subgroup $T \subset B \subset G$ determines a decomposition of the roots $\Phi = \Phi^{+} \cup \Phi^{-}$ and fixes a base of simple roots $\Delta \subseteq \Phi^{+}$. 

A parabolic subgroup of $G$ is a subgroup that contains a Borel. The parabolics which contain $B$ are called the \emph{standard} parabolics. There are finitely many of these and they are indexed by the subsets of $\Delta$. Furthermore, every parabolic subgroup of $G$ is conjugate to a unique standard parabolic. Let $P_{I}$ denote the standard parabolic associated to the subset $I \subseteq \Delta$. It is generated by $T$, the positive root spaces $U_{\alpha}$ for $\alpha \in \Phi^{+}$, and the negative root spaces $U_{\alpha}$ for $\alpha$ in the $\bb{Z}_{\geq 0}$-span of $-I$. The unipotent radical of $P_{I}$ is the normal subgroup $U_{I}$ generated by the positive root spaces which are not in the $\bb{Z}$-span of $I$. The quotient $L_{I} = P_{I}/U_{I}$ is a reductive group whose root system $\Phi(I)$ has a base given by $I$. A Levi subgroup is a lift of $L_{I}$ to $P_{I}$, and it determines a Levi decomposition $P_{I} = L_{I} \ltimes U_{I}$. All Levi subgroups of $P_{I}$ are conjugate, and there is a canonical choice generated by $T$ and the root spaces in the $\bb{Z}$-span of $I$. 

Following \cite[Section 2]{boalch2011riemann} and \cite[page 55]{mumford1994geometric}, a real semisimple element $a \in \frak{g}$ defines a parabolic subgroup 
\[
P(a) := \{ g \in G \ | \ \lim_{z \to 0}z^{a}gz^{-a} \text{ exists in $G$ along any ray} \},
\]
where $z^{a} = \exp( \log(z) a)$. If $a \in \frak{t}$, and $(\alpha, a) \geq 0$ for all $\alpha \in \Phi^{+}$, then $P(a)$ is the standard parabolic associated to the subset $I(a) \subseteq \Delta$, defined by 
\[
I(a) = \{ \alpha \in \Delta \ | \ (\alpha, a) = 0 \}. 
\]
The unipotent radical of $P(a)$ is given by 
\[
U(a) = \{ g \in G \ | \ \lim_{z \to 0}z^{a}gz^{-a} = 1 \text{ along any ray} \},
\]
and a choice of Levi subgroup is given by $C_{G}(a)$, the centraliser of $a$ in $G$. Therefore $P(a) = C_{G}(a) \ltimes U(a)$. Note that the quotient map $\chi : P(a) \to C_{G}(a)$ can be defined by taking a limit  
\[
\chi(g) = \lim_{z \to 0} z^{a} g z^{-a}
\]
along any ray.

\section{Logarithmic connections: Local theory} \label{Firstsection}
In this section, we study the representation theory of $\Pi(\bb{A}, 0) \cong \bb{C} \ltimes \bb{A}$. By Lie's second theorem, these representations are equivalent to logarithmic flat connections on the affine line, and hence to linear ordinary differential equations of the form \ref{LogODE}. In this section, the structure group $G$ is a connected complex reductive group. 

We start by explaining how the invariants of a logarithmic connection arise out of the corresponding $\bb{C} \ltimes \bb{A}$ representation. There are two natural Lie groupoid homomorphisms 
\begin{align*}
\iota : \ &\bb{C} \to \bb{C} \ltimes \bb{A}, \qquad p : \bb{C} \ltimes \bb{A} \to \bb{C} \\
& \lambda \mapsto (\lambda, 0), \qquad  \qquad (\lambda, z) \mapsto \lambda,
\end{align*}
with the property that $p \circ \iota = id_{\bb{C}}$. These induce pullback functors between the categories of representations 
\[
\iota^{*} : \text{Rep}(\bb{C} \ltimes \bb{A}, G) \to \text{Rep}(\bb{C}, G), \qquad p^{*} : \text{Rep}(\bb{C}, G) \to \text{Rep}(\bb{C} \ltimes \bb{A}, G).
\]
Given a representation $\Phi$ of $\bb{C} \ltimes \bb{A}$ on a principal bundle $P$, the pullback $\iota^{*}(\Phi)$ has the form 
\[
\iota^{*}(\Phi) : \bb{C} \to \Aut_{G}(P_{0}), \qquad \lambda \to \exp(\lambda A_{0}),
\]
for a unique Lie algebra element $A_{0} \in \frak{aut}_{G}(P_{0})$. For a representation which corresponds to a differential equation of the form \ref{LogODE}, we have $A_{0} = A(0)$. Hence, $A_{0}$ is called the \emph{residue} of the representation, and is denoted $\Res(\Phi)$. 

The representations in the image of $p^{*}$ are said to be \emph{trivial}, because they are defined on trivial bundles $K \times \bb{A}$, for $G$-torsors $K$, and have the form 
\[
\phi(\lambda, z) = \exp(\lambda A),
\]
for $A \in \frak{aut}_{G}(K)$. These correspond to differential equations of the form \ref{LogODE} for which the connection $1$-form $A(z)$ is constant. 

The \emph{monodromy} of a representation $(P, \Phi)$ at a non-zero point $z \in \bb{A}$ is the $G$-equivariant automorphism $M(z) = \Phi(2 \pi i, z) : P_{z} \to P_{z}$. This represents the failure of a solution to the corresponding differential equation \ref{LogODE} to be single-valued. If we let $z$ vary over the affine line then we get a holomorphic automorphism $M$ of $(P, \Phi)$. At the point $z = 0$, it is related to the residue via the exponential map
\[
M(0) = \Phi(2\pi i, 0) = \exp(2 \pi i \Res(\Phi)). 
\]
If the bundle $P$ is trivial, then the monodromy maps $M(z)$ for different non-zero values of $z$ are all conjugate. The automorphism $M(0)$ may fail to be conjugate to these maps, but it always lies in the closure of their conjugacy class. 

Combining the functors $\iota^{*}$ and $p^{*}$ we get a \emph{linear approximation} functor 
\[
L := p^{*} \circ \iota^{*} : \rep{(\mathbb{C} \ltimes \mathbb{A}, G)} \to \rep{(\mathbb{C} \ltimes \mathbb{A}, G)}.
\]
This takes an arbitrary representation of $\mathbb{C} \ltimes \mathbb{A}$ and outputs the trivial representation determined by its residue. We highlight the following important definition. 
\begin{definition}
A \emph{linearization} of a representation $\Phi$ of $\mathbb{C} \ltimes \mathbb{A}$ is given by an isomorphism of representations 
\[
T : L(\Phi) \to \Phi. 
\]
A linearization $T$ is \emph{strict} if $i^{*}(T) = id_{i^*(\Phi)}$. 
\end{definition}
\begin{remark}
The notion of a strict linearization is well-defined because $i^{*}L = i^*$. Furthermore, $L^2 = L$, so if $T$ is a linearization of a representation, then $T \circ L(T^{-1})$ defines a strict linearization. Hence, linearizations exist if and only if strict ones exist. 
\end{remark}
\begin{remark}
Let $T$ be a linearization of a representation $(P, \Phi)$. For every point $z \in \bb{A}$, this defines a $G$-equivariant map $T_{z} : P_{0} \to P_{z}$. Hence, a linearization may be interpreted as a regularized parallel transport out of the singularity. 
\end{remark}

Consider a linearizable representation $(P, \Phi)$. It's linear approximation is given by ${(P_{0} \times \bb{A}, L(\Phi))}$, where $L(\Phi)(\lambda, z) = \exp( \lambda A)$, and where $A \in \frak{aut}_{G}(P_{0})$ is its residue. The space of linearizations, denoted $\nu(\Phi)$, is a right torsor for the group $\Aut(L(\Phi))$ of automorphisms of $L(\Phi)$. The space of strict linearizations, denoted $\nu_{0}(\Phi)$, is a right torsor for the subgroup $\Aut_{0}(L(\Phi))$ consisting of automorphisms of $L(\Phi)$ which are the identity at the point $0 \in \mathbb{A}$. 

The most important linearizable representations in this paper are those with semisimple monodromy, and it will be useful to have a description of the automorphism group of $L(\Phi)$ in this case. For simplicity of presentation, we choose a trivialization of $P$, and write $L(\Phi)(\lambda, z) = \exp(\lambda S) : \mathbb{C} \ltimes \mathbb{A} \to G$, where $S \in \frak{g}$ is semisimple. Let $S = a + ib$ be the decomposition into real and imaginary parts. An automorphism of $L(\Phi)$ is given by a holomorphic map $g : \bb{A} \to G$ which satisfies 
\[
g(e^{\lambda}z) L(\Phi)(\lambda, z) = L(\phi)(\lambda, z)g(z), 
\]
for all $(\lambda, z) \in \bb{C} \ltimes \bb{A}$. As a result of this equation, an automorphism is determined by its value at a single non-zero point of $\bb{A}$. Hence, we can embed the automorphism group into $G$ via the homomorphism 
\[
\Aut(L(\Phi)) \to G, \qquad g \mapsto g(1). 
\]
We can also restrict an automorphism to the point $0 \in \bb{A}$ to obtain an automorphism of the representation $i^{*}(\Phi)$. This gives rise to the following short exact sequence 
\begin{equation} \label{autexactseq}
1 \to \Aut_{0}(L(\Phi)) \to \Aut(L(\Phi)) \to \Aut(i^{*}(\Phi)) \to 1,
\end{equation}
which has a splitting given by the pullback functor $p^{*}$. 

\begin{proposition}\label{autofrep}
The automorphism groups of $L(\Phi)$ and $\iota^{*}(\Phi)$ are given by 
\[
\Aut_{0}(L(\Phi)) = U(a) \cap C_{G}(\exp(2 \pi i S)), \ \Aut(L(\Phi)) = P(a) \cap C_{G}(\exp(2 \pi i S)), \ \Aut(\iota^{*}(\Phi)) = C_{G}(S), 
\]
where $P(a)$ is the parabolic subgroup of $G$ determined by the real semisimple element $a \in \frak{g}$, $U(a)$ is it's unipotent radical, $C_{G}(S)$ is the centralizer of $S$, and $C_{G}(\exp(2 \pi i S))$ is the centralizer of $\exp(2 \pi i S)$, which is reductive but possibly disconnected. These three automorphism groups are connected. Furthermore, $\Aut(L(\Phi))$ is the parabolic subgroup of $C_{G}(\exp(2 \pi i S))$ determined by $a$, $\Aut_{0}(L(\Phi))$ is its unipotent radical, and $\Aut(\iota^{*}(\Phi))$ is embedded as a Levi subgroup via $p^{*}$.

Given the choice of a maximal torus of $G$ whose Lie algebra contains $S$, we have the identification 
\[
\Aut_{0}(L(\Phi)) = U_{\bb{N}}(S),
\]
where $U_{\bb{N}}(S)$ is the subgroup of $G$ generated by the root spaces for roots that pair with $S$ to give a positive integer. As a result, we have the identification 
\[
\Aut(L(\Phi)) = C_{G}(S) \ltimes U_{\bb{N}}(S). 
\]
\end{proposition} 
\begin{remark}
The group $U_{\bb{N}}(S)$ was also considered in \cite[Section 8.5]{babbitt1983formal} in relation to their classification of logarithmic connections.
\end{remark}
\begin{proof}
An element $g \in \Aut(L(\Phi))$ satisfies 
\begin{equation} \label{defgauge}
g(e^{\lambda}z) = \exp(\lambda S) g(z) \exp(\lambda S)^{-1}.
\end{equation}
Setting $\lambda = 2 \pi i$ we see that $g(1) \in C_{G}(\exp(2 \pi i S)) = C_{G}(\exp(2 \pi i a)) \cap C_{G}(b)$. Setting $\lambda = \log(z)$ and $z = 1$, we get 
\[
g(z) = z^{a} g(1) z^{-a},
\]
for $z \neq 0$. Therefore $\lim_{z \to 0}z^{a} g(1) z^{-a} = g(0)$, and so $g(1) \in P(a)$. Conversely, let \[g_{1} \in P(a) \cap C_{G}(\exp(2 \pi i S)) = C_{G}(S) \ltimes U_{\bb{N}}(S).\] It has a decomposition as $g_{1} = c \prod_{j}\exp(n_{j})$, where $c \in C_{G}(S)$ and $n_{j} \in \frak{g_{\alpha_{j}}}$, for a root $\alpha_{j}$ satisfying $\alpha_{j}(S) \in \bb{N}$. We can therefore define an automorphism of $L(\Phi)$ by setting 
\[
g(z) = z^{S} g_{1} z^{-S} = c \prod_{j} \exp(z^{\alpha_{j}(S)} n_{j}).
\]
The rest of the proof is straightforward and is left to the reader.
\end{proof}


Certain representations cannot be linearized. We can see this by observing that for trivial representations, the exponential of the residue is conjugate to the monodromy. This is a property that is preserved by gauge transformations, and hence it holds for all linearizable representations. Therefore, representations that do not satisfy this property cannot be linearized. This turns out to be the only obstruction to linearizability, as the following proposition shows. 

\begin{proposition} \label{sslinearize}
A representation $\Phi$ on a principal bundle $P$ is linearizable if and only if its monodromy $\Phi(2 \pi i , 1)$ is conjugate to $\exp( 2 \pi i \Res(\Phi))$. In particular, since the conjugacy classes of semisimple elements are closed, representations with semisimple monodromy can always be linearized.  
\end{proposition}
\begin{proof}
One direction follows from the discussion above. Hence it remains for us to show that a representation $\Phi$ can be linearized if its monodromy is conjugate to the exponential of its residue. To simplify, we trivialize the $G$-bundle, so that the representation is determined by a groupoid homomorphism $\phi : \mathbb{C} \ltimes \mathbb{A} \to G$. The residue is then a Lie algebra element $A \in \mathfrak{g}$ such that $\phi(\lambda, 0) = \exp(\lambda A)$, and a linearization is given by a holomorphic map $g : \mathbb{A} \to G$ such that 
\begin{equation} \label{goal}
g(e^{\lambda} z) e^{\lambda A} g(z)^{-1} = \phi(\lambda, z),
\end{equation}
for all $(\lambda, z)$. Note that it suffices to find $g$ in a neighbourhood of the point $0 \in \mathbb{A}$, since equation \ref{goal} will automatically extend $g$ to the entire affine line. 

Let $M(z) = \phi(2\pi i, z)$ be the monodromy, and let $\mathcal{C} \subseteq G$ be the conjugacy class of $\exp(2 \pi i A)$, which is an embedded submanifold of $G$. By assumption, $M : \mathbb{A} \to \mathcal{C}$. It is therefore possible to find a disc $D(0,r)$ of radius $r$ around $0 \in \mathbb{A}$, and a holomorphic map $h : D(0,r) \to G$ such that $h(0) = 1$ and $h(z)e^{2 \pi i A}h(z)^{-1} = \phi(2\pi i, z)$. As a result, the formula 
\[
\psi(\lambda, z) = h(e^{\lambda} z)^{-1} \phi(\lambda, z) h(z),
\]
defines a homomorphism $\psi : (\mathbb{C} \ltimes \mathbb{A})|_{D(0,r)} \to G$ such that $\psi(2 \pi i, z) = e^{2\pi i A}$ and $\psi(\lambda, 0) = e^{\lambda A}$. In fact, because of the equation $(\lambda, z)  (2 \pi i , z) = (2 \pi i, e^{\lambda}z)  (\lambda, z)$ in the groupoid, the image of $\psi$ is contained in $C_{G}(e^{2 \pi i A})$, the centralizer of $e^{2 \pi i A}$ in $G$. This allows us to define an action of $(\mathbb{C} \ltimes \mathbb{A})|_{D(0,r)}$ on $D(0,r) \times C_{G}(e^{2 \pi i A})$ via the formula 
\[
(\lambda, z) \ast (z, g) = (e^{\lambda} z, \psi(\lambda, z) g e^{-\lambda A}).
\]
But now note that $(2 \pi i, z) \ast (z,g) = (z, g)$, implying that this action descends to an action of $(\mathbb{C}^{*} \ltimes \mathbb{A})|_{D(0,r)}$. The point $(0,1) \in D(0,r) \times C_{G}(e^{2 \pi i A})$ is a fixed point of the action. We can hence use a Bochner linearization to find an equivariant identification between a neighbourhood of this fixed point and a neighbourhood of the origin in $T_{(0,1)}(D(0,r) \times C_{G}(e^{2 \pi i A}))$. This identification can furthermore be chosen to be compatible with the projection to $D(0,r)$. 

The $\mathbb{C}^{*}$-action on $T_{(0,1)}(D(0,r) \times C_{G}(e^{2 \pi i A}))$ is linear, and the Lie algebra of $C_{G}(e^{2 \pi i A})$, which we denote $\mathfrak{c}$, is an invariant subspace. Therefore, we have a short exact sequence of representations
\[
0 \to \mathfrak{c} \to T_{(0,1)}(D(0,r) \times C_{G}(e^{2 \pi i A})) \to \mathbb{C} \to 0,
\]
where $\mathbb{C} = T_{0}D(0,r)$ carries the standard weight $1$ action. Since $\mathbb{C}^{*}$-representations are completely reducible, we can split this sequence to produce a $\mathbb{C}^{*}$-invariant subspace. Using the equivariant identification, and shrinking $r$ as necessary, we therefore obtain a holomorphic map $k : D(0, r) \to C_{G}(e^{2 \pi i A})$ which satisfies 
\[
(\lambda, z) \ast (z, k(z)) = (e^{\lambda}z, k(e^{\lambda}z)),
\]
for $(\lambda, z) \in (\mathbb{C} \ltimes \mathbb{A})|_{D(0,r)}$. The holomorphic map $hk : D(0,r) \to G$ provides the desired linearization.

\end{proof}

\subsection{Jordan-Chevalley decomposition} \label{SecJCdec}
In this section, we explain that representations of $\bb{C} \ltimes \bb{A}$ admit a Jordan-Chevalley decomposition. In a sense, this is an integrated version of the result of \cite{levelt1975jordan} (generalized to arbitrary $G$ in \cite{babbitt1983formal}). We start by showing that the ordinary multiplicative JC decomposition can be applied to the monodromy of a representation. 

\begin{lemma}
Let $(P, \Phi)$ be a representation of $\bb{C} \ltimes \bb{A}$, with monodromy $M$, and let $M_{s}(z)$ and $M_{u}(z)$ denote the semisimple and unipotent components of $M(z)$, respectively. Then $M_{s}$ and $M_{u}$ define holomorphic automorphisms of $(P, \Phi)$. 
\end{lemma}
\begin{proof}
We start by trivializing the principal bundle $P$ and choosing a faithful representation of $G$, so that we are working with a representation of the form $\phi : \bb{C} \ltimes \bb{A} \to GL(N, \bb{C})$. This is valid since the JC decomposition is preserved by morphisms of reductive groups. 

The monodromy is an automorphism of $\phi$, and so satisfies 
\begin{equation} \label{monodromyeqproof}
M(e^{\lambda}z) = \phi(\lambda, z) M(z) \phi(\lambda, z)^{-1},
\end{equation}
for all $(\lambda, z) \in \mathbb{C} \ltimes \bb{A}$. Since conjugation preserves the $JC$ decomposition, Equation \ref{monodromyeqproof} implies that $M_{s}(z)$ and $M_{u}(z)$ vary holomorphically over $\bb{A} \setminus \{ 0\}$, and that their conjugacy classes remain constant over this locus. Hence, there exists a polynomial $p$, without constant term, such that $M_{s}(z) = p(M(z))$, for $z \neq 0$ \cite[Section 15]{humphreys2012linear}. In fact, this equality remains valid at $z = 0$ because semisimple elements have closed conjugacy classes. Furthermore, $p(M(z))$ is clearly holomorphic over $\bb{A}$, and so $M_{s}$ is likewise holomorphic. Then, $M_{u}(z) = M_{s}(z)^{-1}M(z)$ is also holomorphic at $0$. Finally, Equation \ref{monodromyeqproof} and the uniqueness of the JC decomposition imply that $M_{s}$ and $M_{u}$ define automorphisms of $\phi$.  
\end{proof}

We now recall the notion of groupoid $1$-cocycles, which we will use to deform representations. A $1$-cocycle for $\bb{C} \ltimes \bb{A}$, valued in a representation $(P, \Phi)$, is defined to be a holomorphic section $\sigma$ of the bundle of groups $t^{*}\Aut_{G}(P)$ over $\bb{C} \ltimes \bb{A}$, which satisfies the following cocycle condition 
\[
\sigma(\mu, e^{\lambda}z) \Phi(\mu, e^{\lambda}z) \sigma(\lambda, z) = \sigma(\mu + \lambda, z) \Phi(\mu, e^{\lambda}z),
\]
for all $(\mu, \lambda, z) \in \bb{C} \times \bb{C} \times \bb{A}$. Given a representation $(P, \Phi)$, and a $1$-cocycle $\sigma$, then $(P, \sigma \circ \Phi)$ defines a new representation. 

\begin{proposition} \label{UntwistingcocycleProp}
Let $\Phi$ be a representation of $\bb{C} \ltimes \bb{A}$ on a principal bundle $P$. Then the following formula defines a holomorphic groupoid $1$-cocycle
\[
\sigma_{\Phi}(\lambda, z) = \exp(\frac{-\lambda}{2 \pi i } \log(M_{u}(e^{\lambda}z))). 
\]
Furthermore, $\sigma_{\Phi}$ consists of unipotent automorphisms of $P$.
\end{proposition}
\begin{proof}
To see that $\sigma_{\Phi}$ is holomorphic, we first choose a trivialization of $P$ and a faithful representation of $G$, so that $M_{u}$ is represented by a holomorphic family of unipotent matrices. The holomorphicity of $\sigma_{\Phi}$ then follows from the fact that the matrix logarithm is well-defined for unipotent matrices and is given by a polynomial.

To check that $\sigma_{\Phi}$ satisfies the cocycle condition, we start with the fact that $M_{u}$ defines an automorphism of $\Phi$, which is encoded by the equation
\[
M_{u}(e^{\lambda}z) = \Phi(\lambda, z) M_{u}(z) \Phi(\lambda, z)^{-1}.
\]
Then using the fact that both the exponential and the logarithm respect this equation, we get 
\[
\Phi(\mu, e^{\lambda}z) \sigma_{\Phi}(\lambda, z) \Phi(\mu, e^{\lambda}z)^{-1} =  \exp(\frac{-\lambda}{2 \pi i } \log M_{u}(e^{\mu + \lambda}z)),
\]
so that 
\begin{align*}
    \sigma_{\Phi}(\mu, e^{\lambda}z) \Phi(\mu, e^{\lambda}z) \sigma_{\Phi}(\lambda, z)  \Phi(\mu, e^{\lambda}z)^{-1} &= \exp(\frac{-\mu}{2 \pi i } \log M_{u}(e^{\mu + \lambda}z))\exp(\frac{-\lambda}{2 \pi i } \log M_{u}(e^{\mu + \lambda}z)) \\
    &= \sigma_{\Phi}(\mu + \lambda, z).
\end{align*}
\end{proof}

Proposition \ref{UntwistingcocycleProp} leads to a Jordan-Chevalley decomposition for representations of $\bb{C} \ltimes \bb{A}$. Indeed, given a representation $\Phi$, the resulting representation $\Phi_{s} := \sigma_{\Phi} \circ \Phi$ has semisimple monodromy given by $M_{s}$, and the unipotent map $M_{u}$ is an automorphism of $\Phi_{s}$. This decomposition extends both the multiplicative Jordan-Chevalley decomposition of the monodromy, as well as the additive Jordan-Chevalley decomposition of the residue. We can make this JC decomposition functorial by introducing a category $\mathcal{JC}$ whose objects are tuples $(P, \Psi, U)$, where $\Psi$ is a representation of $\bb{C} \ltimes \bb{A}$ with semisimple monodromy on a principal bundle $P$, and $U$ is a unipotent automorphism of $(P, \Psi)$. 

\begin{theorem}\label{JCdecomp}
There is an isomorphism of categories $D : \rep{(\mathbb{C} \ltimes \mathbb{A}, G)} \to \mathcal{JC}$ defined on objects by sending a representation $(P, \Phi)$ to the tuple $(P, \sigma_{\Phi} \circ \Phi, M_{u})$, and defined on morphisms to be the identity. The inverse functor $D^{-1}$ sends a tuple $(P, \Psi, U)$ to the representation $(P, \tau_{U} \circ \Psi)$, where $\tau_{U}$ is the $1$-cocycle defined by the equation
\[
\tau_{U}(\lambda, z) = \exp(\frac{\lambda}{2 \pi i} \log( U(e^{\lambda} z))). 
\]
\end{theorem}

Given a representation $(P, \Phi)$, we call $\Aut_{0}(L(\Phi_{s}))$ the \emph{resonance} group of $\Phi$. A representation $\Phi$ is said to be \emph{resonant} when its resonance group is non-trivial. By Proposition \ref{autofrep}, this occurs when the semisimple part of the residue of $\Phi$ pairs with a root to give a positive integer. When the structure group is $GL(n, \bb{C})$, this is equivalent to the condition that two eigenvalues of the residue differ by an integer. 

The phenomenon of resonance is the main source of subtlety in the classification of logarithmic connections. Indeed, the failure of linearizability for some representations is often attributed to resonance. However, we have seen in Proposition \ref{sslinearize} that a representation fails to be linearizable precisely when its monodromy is not conjugate to the exponential of its residue. In fact, the failure of linearizability is due to the interaction between the resonance group and the unipotent part of the monodromy. 

To see this, consider an object $(P, \Phi_{s}, U)$ of $\cal{JC}$ corresponding to a representation $\Phi$, and consider the following short exact sequence of groups 
\[
1 \to \Aut_{0}(L(\Phi_{s})) \to \Aut(L(\Phi_{s})) \to \Aut(\iota^{*}(\Phi_{s})) \to 1.
\]
By Proposition \ref{autofrep}, this sequence is isomorphic to the short exact sequence associated to a parabolic subgroup. This sequence has a canonical splitting given by the functor $p^{*}$, and the image of this splitting is a Levi subgroup. The linear approximation of $\Phi$ can be expressed in the category $\cal{JC}$ as $(P_{0} \times \bb{A}, L(\Phi_{s}), \iota^{*}(U))$, where $\iota^{*}(U)$ is viewed as an element of the Levi subgroup. A linearization of $\Phi$ is therefore equivalent to a linearization $T$ of $\Phi_{s}$, such that $T^{-1}UT = \iota^{*}(U)$. By first choosing an arbitrary linearization of $\Phi_{s}$, we can view $U$ as an element of $\Aut(L(\Phi_{s}))$ which projects to $i^{*}(U)$. Then, a linearization is equivalent to an element of $\Aut(L(\Phi_{s}))$ which conjugates $U$ to the element $\iota^{*}(U)$ in the Levi subgroup. It is clear that a linearization exists if either $U = 1$ or the resonance group is trivial. In general, if the resonance group is non-trivial, then $\iota^{*}(U)$ and $U$ may fail to be conjugate, in which case a linearization does not exist. 

%

\subsection{Classification} \label{classification}
In this section, we prove a functorial classification result for $\bb{C} \ltimes \bb{A}$ representations. Our approach to classification makes use of the JC decomposition of Theorem \ref{JCdecomp}. Namely, we first apply the functor $D$ from Theorem \ref{JCdecomp} to extract a representation $\Phi_{s} = \sigma_{\Phi} \circ \Phi$ with semisimple monodromy. Then, we take the space of strict linearizations 
\[
\nu_{0}(\Phi_{s}) \subset \Hom{L(\Phi_{s}), \Phi_{s}}  \subseteq \spc{Hom}_{G}(P_{0},P_{1}),
\]
which is non-empty by Proposition \ref{sslinearize}, and which is a torsor for the resonance group $\Aut_{0}(L(\Phi_{s}))$. The data that we use to classify the representation consists of $(M(\Phi), \nu_{0}(\Phi_{s}), \Res(\Phi))$, where $M(\Phi) = \Phi(2 \pi i, 1)$ is the monodromy at the point $1$. 
\subsubsection{The category of generalized monodromy data} \label{Definingthecat}
We define a category $F((\bb{C},0), G)$ of generalized monodromy data, which we will use to classify representations of $\bb{C} \ltimes \bb{A}$. To make sense of the definition, we start with a few simple observations. Let $K_{1}$ and $K_{0}$ be right $G$-torsors. The space of $G$-equivariant maps $\spc{Hom}_{G}(K_{0},K_{1})$ is then a bi-torsor for $\Aut_{G}(K_{1})$ acting on the left and $\Aut_{G}(K_{0})$ acting on the right. Let $A \in \frak{aut}_{G}(K_{0})$, which admits a JC decomposition $A = S + N$, for $S$ semisimple and $N$ nilpotent. The element $S$ defines two subgroups 
\[
U_{\bb{N}}(S) \subset C_{\Aut_{G}(K_{0})}(S) \ltimes U_{\bb{N}}(S) \subseteq \Aut_{G}(K_{0}).
\]
These groups were identified in Proposition \ref{autofrep} with the automorphism groups of the trivial representation $\exp(\lambda S)$ on $K_{0} \times \bb{A}$. In particular, $U_{\bb{N}}(S)$ is the resonance group of $\exp(\lambda S)$. Next, let  \[ \nu_{0} \subset \spc{Hom}_{G}(K_{0},K_{1}) \] be a reduction of structure to $U_{\bb{N}}(S)$, which is acting on the right. This extends to a right $C_{\Aut_{G}(K_{0})}(S) \ltimes U_{\bb{N}}(S)$-torsor denoted $\nu \subseteq \spc{Hom}_{G}(K_{0},K_{1})$. By taking the stabilizers of $\nu_{0}$ and $\nu$ for the left action, we obtain subgroups $\text{St}(\nu_{0}) \subset \text{St}(\nu)$ of $\Aut_{G}(K_{1})$. The groups $\text{St}(\nu_{0})$ and $\text{St}(\nu)$ are non-canonically isomorphic to $U_{\bb{N}}(S)$ and $C_{\Aut_{G}(K_{0})}(S) \ltimes U_{\bb{N}}(S)$, respectively. Therefore, $\text{St}(\nu_{0})$ is normal in $\text{St}(\nu)$, and there is a canonical short exact sequence 
\begin{equation} \label{canonicalprojection}
1 \to \text{St}(\nu_{0}) \to \text{St}(\nu) \xrightarrow{\chi} C_{\Aut_{G}(K_{0})}(S) \to 1.
\end{equation}
The objects of $F((\bb{C},0), G)$ are defined to be tuples $(M, K_{1}, \nu_{0}, K_{0}, A)$, where $K_{1}, K_{0}, \nu_{0}$ and $A$ are as above, and $M \in St(\nu)$ satisfies 
\[
\chi(M) = \exp(2 \pi i A). 
\]
The morphisms in $F((\bb{C},0),G)$ are pairs of maps between the underlying $G$-torsors that preserve all the structure.

\subsubsection{Equivalence of categories}
The preceding discussion, combined with Proposition \ref{sslinearize} and Proposition \ref{autofrep}, implies that we have a well-defined functor 
\begin{align} \label{Leveltfunc}
\mathcal{L} : \Rep(\mathbb{C} \ltimes \mathbb{A}, G) \to F((\mathbb{C},0), G), \qquad (P, \Phi) \mapsto (M(\Phi),P_{1}, \nu_{0}(\Phi_{s}),P_{0}, \Res\Phi).
\end{align}
In the opposite direction, we define a functor 
\begin{equation} \label{Reesfunc}
\mathcal{R} : F((\mathbb{C},0), G) \to \Rep(\mathbb{C} \ltimes \mathbb{A}, G),
\end{equation}
which will provide an equivalence of categories. Again, we use the JC decomposition of Theorem \ref{JCdecomp}. Given an object $(M, K_{1}, \nu_{0}, K_{0}, A)$ of $F((\mathbb{C},0), G)$, we first construct a representation $\Phi_{s}$ with semisimple monodromy. The group $U_{\bb{N}}(S)$ is identified with the resonance group of the trivial representation $\exp(\lambda S)$, and therefore it acts on the trivial bundle $K_{0} \times \bb{A}$. We use this action to construct a principal bundle 
\[
\nu_{0} \otimes K_{0} = (\nu_{0} \times K_{0} \times \bb{A})/U_{\bb{N}}(S).
\]
The representation $\Phi_{s}$ is defined on this bundle via the formula
\[
\Phi_{s}(\lambda, z)(h, k, z) = (h, \exp(\lambda S)(k), e^{\lambda}z),
\]
for $(h,k,z) \in \nu_{0} \times K_{0} \times \bb{A}$. It is clear that $\Phi_{s}$ has semisimple monodromy. Next, using the unipotent component $M_{u}$ of $M$ and the nilpotent component $N$ of $A$, we define an automorphism of $\Phi_{s}$ 
\[
U : (h, k, z) \mapsto (M_{u} h \exp(-2\pi i N), \exp(2\pi i N)(k), z). 
\]
We therefore have an object of $\cal{JC}$, and so we define $\cal{R}(M, K_{1}, \nu_{0}, K_{0}, A) = D^{-1}(\nu_{0} \otimes K_{0}, \Phi_{s}, U)$. Given a morphism in $F((\bb{C},0), G)$
\[
(T_{1}, T_{0}) : (M, K_{1}, \nu_{0}, K_{0}, A) \to (M', K'_{1}, \nu'_{0}, K'_{0}, A'), 
\]
we define the morphism $\cal{R}(T_{1}, T_{0}): \nu_{0} \otimes K_{0} \to \nu'_{0} \otimes K'_{0}$ via the formula 
\[
(h, k, z) \mapsto (T_{1}hT_{0}^{-1}, T_{0}(k), z). 
\]
This defines the functor $\mathcal{R}$. 

\begin{theorem} \label{RH}
The pair of functors $(\mathcal{L}, \mathcal{R})$ forms an adjoint equivalence of categories 
\[
\Rep(\bb{C} \ltimes \bb{A}, G) \cong F((\bb{C},0), G). 
\]
\end{theorem}
\begin{proof}
We will sketch the constructions of the natural isomorphisms $\mathcal{L} \mathcal{R} \cong 1$ and $\mathcal{R} \mathcal{L} \cong 1$. The remainder of the proof is then straightforward. First we construct the natural isomorphism $\eta : \cal{R} \cal{L} \to 1$. Given a representation $(P, \Phi)$, the representation $\cal{R} \cal{L}(P, \Phi)$ lives on the bundle 
\[
(\nu_{0}(\Phi_{s}) \times L(P) )/ \Aut_{0}(L(\Phi_{s})).
\]
Hence, using the fact that $\nu_{0}(\Phi_{s}) \subset \Hom{L(P), P}$, we define $\eta_{(P, \Phi)}$ to be the evaluation map to $P$. Second, we construct the natural isomorphism $\epsilon : \cal{L} \cal{R} \to 1$. The two $G$-torsors in the object $\cal{L}\cal{R}(M, K_{1}, \nu_{0}, K_{0}, A)$ are $(\nu_{0} \otimes K_{0})|_{1}$ and $(\nu_{0} \otimes K_{0})|_{0}$. Because the group $U_{\bb{N}}(S)$ acts trivially on the fibre above $0$ we have 
\[
(\nu_{0} \otimes K_{0})|_{0} = (\nu_{0}/U_{\bb{N}}(S)) \times K_{0} \cong K_{0}. 
\]
And because $\nu_{0} \subset \spc{Hom}_{G}(K_{0},K_{1})$, we get an evaluation map 
\[
(\nu_{0} \otimes K_{0})|_{1} \to K_{1}. 
\]
These two morphisms make up the component $\epsilon_{(M, K_{1}, \nu_{0}, K_{0}, A)}$.
\end{proof}

At the cost of a non-canonical equivalence, we can give a more explicit presentation of $F((\bb{C},0), G)$ by taking the full subcategory generated by the objects whose underlying $G$-torsors are trivial. Suppressing the trivial $G$-torsors from the notation, the set of objects of this subcategory is given by the set $\cal{B}$ of tuples $(M, hU_{\bb{N}}(S), A = S + N)$, where $hU_{\bb{N}}(S)$ is a left coset, and such that $M \in h (C_{G}(S) \ltimes U_{\bb{N}}(S)) h^{-1}$, and $\chi(h^{-1} M h) = \exp(2 \pi i A)$. The morphisms in this category with a fixed source object are given by $G \times G$, where a pair $(k,g)$ corresponds to the morphism $(T_{1}, T_{0}) = (kg, g)$. The pair $(k,g)$ acts on an object as follows: \[(k,g) \ast (M, hU_{\bb{N}}(S), A) = ( (kg)M(kg)^{-1}, k(ghg^{-1})U_{\bb{N}}(Ad(g)(S)), Ad(g)(A)).\] Therefore we have the equivalence
\[
F((\bb{C},0), G) \cong (G \rtimes G) \ltimes \cal{B},
\]
where the right-hand side is an action groupoid, and $G \rtimes G$ is the semi-direct product for the conjugation action of $G$ on itself. 

Now consider the full subcategory where the residue is fixed to be $A \in \frak{g}$. The objects of this sub-category can be taken to have the form $(M, U_{\bb{N}}(S), A = S + N)$, where only the element $M$ is allowed to vary. This set of objects is parametrised by $T = \{ M \in C_{G}(S) \ltimes U_{\bb{N}}(S) \ | \ \chi(M) = \exp(2 \pi i A) \}$. Therefore, the subcategory is equivalent to the action groupoid
\[
(C_{G}(A) \ltimes U_{\bb{N}}(S)) \ltimes T,
\]
where $C_{G}(A)$ is the centralizer of $A$. By taking isomorphism classes, we recover the following classification result. 
\begin{corollary} \label{EquivNormalForm}
There is a bijection between isomorphism classes of flat logarithmic $G$-connections on the affine line whose residue is conjugate to a fixed element $A = S + N \in \frak{g}$, and the set of elements $M \in  C_{G}(S) \ltimes U_{\bb{N}}(S)$ such that $\chi(M) = \exp( 2 \pi i A)$, modulo conjugation by $(C_{G}(A) \ltimes U_{\bb{N}}(S))$. 
\end{corollary}

\begin{remark}
We can write an element $M \in T$ in its JC decomposition as $\exp(2 \pi i S) \exp(2 \pi i N')$, where $N' = \sum_{i \geq 0} N_{i}$ is nilpotent, and $[S,N_{i}] = iN_{i}$. Applying the functor $\cal{R}$ from Theorem \ref{RH} and taking the associated differential equation yields the connection $1$-form
\[
A(z) = S + \sum_{i \geq 0} z^{i}N_{i},
\]
which is in Levelt normal form. Therefore, we obtain a new proof of the Levelt normal form. Furthermore, Corollary \ref{EquivNormalForm} can be understood as describing the holomorphic equivalence of different Levelt normal forms. This recovers the result of \cite{kleptsyn2004analytic} in the case $G = GL(n,\mathbb{C})$, as well as the result \cite[Theorem 8.5]{babbitt1983formal} for complex reductive $G$.  
\end{remark}

There is a redundancy in the definition of the category $F((\bb{C},0), G)$. Let $E((\bb{C},0), G)$ be the category whose objects $(M, K_{1}, \mu_{0}, K_{0}, A)$ are defined in the same way as $F((\bb{C},0), G)$, except that $\mu_{0}$ is a right $U(a)$-torsor, where $A = a + ib + N$ is the decomposition with $a, b$ real semisimple, $N$ nilpotent, and all mutually commuting. There is an obvious functor from $F((\bb{C},0), G)$ to $E((\bb{C},0), G)$, which replaces the $U_{\bb{N}}(S)$-torsor $\nu_{0}$ with its extension to a $U(a)$-torsor. Conversely, given the $U(a)$-torsor $\mu_{0}$ from an object of $E((\bb{C},0), G)$, there is a natural way to reduce it to a $U_{\bb{N}}(S)$-torsor, for $S = a + ib$. Namely, we define 
\[
\nu_{0} = \{ h \in \mu_{0} \ | \ h^{-1}M_{s}h = \exp(2 \pi i S) \},
\]
where $M_{s}$ is the semisimple component of $M$. Hence, the categories $F((\bb{C},0),G)$ and $E((\bb{C},0), G)$ are equivalent. 

By looking at the isomorphism classes of objects in $E((\bb{C},0), G)$ we recover the classification result of \cite[Theorem A]{boalch2011riemann}. Let $\mathcal{C} \subseteq C_{G}(a)$ be the conjugacy class of $\exp(2\pi i A)$. By \cite[Lemma 1]{boalch2011riemann}, this canonically determines a conjugacy class in the reductive quotient of any parabolic subgroup of $G$ conjugate to $P(a)$. We use the same notation $\mathcal{C}$ to denote all of these conjugacy classes. 

\begin{corollary}\cite[Theorem A]{boalch2011riemann}
There is a bijection between isomorphism classes of flat logarithmic $G$-connections on the affine line whose residue is conjugate to a fixed element $A = a + ib + N \in \frak{g}$, and conjugacy classes of pairs $(M, P)$, where $P$ is a parabolic conjugate to $P(a)$, and $M \in P$ such that $\chi(M) \in \mathcal{C}$. 
\end{corollary}
\begin{proof}
Let $S$ be the collection of pairs $(M, P)$ as described above. The group $G$ acts on $S$ by conjugation, and we can define an action groupoid $G \ltimes S$. Let $Q$ be the set of pairs $(M, hU(a))$ such that $M \in hP(a)h^{-1}$ and $\chi(h^{-1}Mh) = \exp(2\pi i A)$. The group $G \rtimes C_{G}(A)$ acts on $Q$, and the corresponding action groupoid is equivalent to the full subcategory of $E((\bb{C},0), G)$ consisting of elements whose residue is conjugate to $A$. By Theorem \ref{RH}, this category is equivalent to the category of logarithmic connections with residue conjugate to $A$. Therefore, it suffices to provide an equivalence between $G \ltimes S$ and $(G \rtimes C_{G}(A) ) \ltimes Q$. There is a natural functor given as follows
\begin{align*}
F : (G \rtimes C_{G}(A) ) \ltimes Q &\to G \ltimes S \\
 ( (k,g), (M, hU(a)) ) &\mapsto (kg, (M, hP(a)h^{-1})),
\end{align*}
and it is straightforward to see that it is essentially surjective. To show that $F$ is fully faithful, it suffices to show that it induces an isomorphism between stabilizer groups. But the stabilizer of $(M, U(a))$ is $C_{G}(M) \cap (U(a) \rtimes C_{G}(A))$, the stabilizer of $(M, P(a))$ is $C_{G}(M) \cap P(a) = C_{G}(M) \cap ( U(a) \rtimes C_{G}(a) )$, and these are equal. 
\end{proof}
\begin{remark}
Theorem \ref{RH}, stated in terms of the category $E((\mathbb{C}, 0), G)$, can also be used to recover the equivalence between logarithmic connections and filtered local systems from \cite{simpson1990harmonic, boalch2011riemann}. 
\end{remark}

\section{Morita equivalence}
Morita equivalence is an equivalence relation between Lie groupoids that is weaker than isomorphism. The approach that we take here, originally introduced in \cite{hilsum1987morphismes, haefliger1984groupoides, pradines1989morphisms}, is based on the use of principal groupoid bibundles.

Let $\cal{G}$ be a Lie groupoid over $M$, with target and source maps denoted by $t$ and $s$ respectively. This can act on a manifold $X$ equipped with a \emph{moment map} $\phi : X \to M$. Namely, a left action of $\cal{G}$ on $X$ is given by a map 
\[
\theta : \cal{G} \times_{s,\phi} X \to X, \qquad (g, x) \mapsto g.x, 
\]
subject to the compatibility conditions that $\phi(g.x) = t(g)$, $g.(g'.x) = (gg').x$ where defined, and $\epsilon(\phi(x)).x = x$, where $\epsilon : M \to \cal{G}$ is the embedding of identity arrows. A representation of $\cal{G}$ on a principal $G$-bundle $P \to M$ is equivalent to a left action of $\cal{G}$ on $P$ such that the action on the fibres is $G$-equivariant.

A left $\cal{G}$-bundle is defined by the data of a surjective submersion $\pi : Q \to N$ and a left action of $\cal{G}$ on $Q$ which preserves the fibers of $\pi$. The bundle is \emph{principal} if the action of $\cal{G}$ is free and fibrewise transitive, meaning that the following map is an isomorphism
\[
\cal{G} \times_{M} Q \to Q \times_{N} Q, \qquad (g,q) \mapsto (g.q, q). 
\]
The notions of right actions and right principal bundles are defined analogously. 

\begin{definition}\label{ME}
A \emph{Morita equivalence} between Lie groupoids $\cal{G}$ over $M$ and $\cal{H}$ over $N$ is a bi-principal $(\cal{G}, \cal{H})$ bi-bundle. This consists of a span
\[
\xymatrixcolsep{3pc} \xymatrix { & Q \ar[dl]^{p}\ar[dr]_{q}& \\ 
M & & N
}
\]
such that
\begin{enumerate}
\item $q : Q \to N$ is a left principal $\cal{G}$-bundle with moment map $p$;
\item $p : Q \to M$ is a right principal $\cal{H}$-bundle with moment map $q$;
\item the actions of $\cal{G}$ and $\cal{H}$ are compatible in the sense that $(g.z).h = g.(z.h)$, where $z \in Q, g \in \cal{G}, h \in \cal{H}$ and both actions are defined.
\end{enumerate}
We say that the groupoids $\cal{G}$ and $\cal{H}$ are \emph{Morita equivalent}. 
\end{definition}

The significance of Morita equivalences in this paper is that they provide a method for constructing equivalences between categories of groupoid representations.

\begin{proposition} \label{repeq}
Let $G$ be a Lie group. A Morita equivalence $Q$ between groupoids $\cal{G}$ and $\cal{H}$ induces an equivalence between their categories of representations
\[
\Rep(\cal{H}, G) \cong  \Rep(\cal{G},G).
\]
\end{proposition}
\begin{proof}
The equivalence is given by a pull-push procedure. We describe the functor on objects in one direction. Consider a representation of $\cal{H}$ on the principal $G$-bundle $P \to N$. Pulling back along $q$ we get the principal bundle $q^{*}P = Q \times_{N} P$ over $Q$. This has a right $\cal{H}$ action given by 
\[
(q,v).h := (q.h, h^{-1}.v),
\]
and a left $\cal{G}$ action given by 
\[
g.(q,v) = (g.q, v).
\]
These actions are compatible, $G$-equivariant, and lift the respective actions on $Q$. Hence, taking the quotient by $\cal{H}$ gives a principal $G$-bundle $Q \otimes P := (Q \times_{N} P)/\cal{H}$ over $M$, and it inherits a left action of $\cal{G}$, making it into a representation of $\cal{G}$. 
%
%
\end{proof}
\begin{example}
The fundamental groupoid $\Pi(M)$ of a manifold $M$ is Morita equivalent to the fundamental group $\pi(M,m)$ based at a point $m \in M$. The Morita equivalence is given by the universal cover $\widetilde{M}$, defined as the space of homotopy classes of paths starting at the point $m$. The groupoid actions on $\widetilde{M}$ are given by concatenation of paths. Applying Proposition \ref{repeq} yields the equivalence of categories 
\[
\Rep(\Pi(M), G) \cong \Rep(\pi(M,m), G).
\]
When combined with the equivalence $\Rep(\Pi(M), G) \cong \Rep(TM, G)$ given by Lie's second theorem, this recovers the classical Riemann-Hilbert correspondence. 
\end{example}

Our main tool for constructing Morita equivalences is the following result, whose proof is left to the reader. 

\begin{lemma} \label{MRL} Let $\cal{G} \rightrightarrows M$ be a Lie groupoid, with target and source maps denoted $t$ and $s$ respectively, and let $N \subseteq M$ be an embedded submanifold. If $t|_{s^{-1}(N)} : s^{-1}(N) \to M$ is a surjective submersion,
then $\cal{G}|_{N} = t^{-1}(N) \cap s^{-1}(N)$ is a Lie subgroupoid of $\cal{G}$, which is Morita equivalent to $\cal{G}$ via the following Morita equivalence 
\[
\xymatrixcolsep{3pc} \xymatrix { { \cal{G}|_{N}} \ar@<-0.5ex>[d]\ar@<0.5ex>[d]  &
t^{-1}(N)\ar[dl]^{t}\ar[dr]_{s} &
{\cal{G}} \ar@<-0.5ex>[d]\ar@<0.5ex>[d]\\
N & & M
 }
\]
\end{lemma}
It will be useful to have a criterion, phrased in terms of the Lie algebroid of $\cal{G}$, for checking that a submanifold $N \subseteq M$ satistifies the condition of Lemma \ref{MRL}. To this end, we recall a few definitions. First let $A \to M$ be the Lie algebroid of $\cal{G}$, and denote its anchor $\rho: A \to TM$. We say that a submanifold $N \subseteq M$ is \emph{transverse} to $A$ if $T_{z}M = T_{z}N + \im \rho_{z}$ for every point $z \in N$. Second, two points $z_{1}, z_{2} \in M$ lie in the same \emph{orbit} of $\cal{G}$ if there is an arrow $g \in \cal{G}$ such that $s(g) = z_{1}$ and $t(g) = z_{2}$. This defines an equivalence relation on $M$, and $M$ is partitioned into the orbits of $\cal{G}$. 

\begin{proposition}[Criterion for Morita equivalent subgroupoid] \label{criterion} Let $\cal{G} \rightrightarrows M$ be a Lie groupoid with Lie algebroid $A \to M$, and let $N \subseteq M$ be an embedded submanifold. If $N$ intersects every orbit of $\cal{G}$ and is transverse to $A$, then $\cal{G}|_{N}$ is a Lie subgroupoid of $\cal{G}$, which is Morita equivalent to $\cal{G}$.
\end{proposition}
\begin{proof}It suffices to check the condition of Lemma \ref{MRL}, namely that $t|_{s^{-1}(N)} : s^{-1}(N) \to M$ is a surjective submersion. That this map is surjective follows from the fact that $N$ intersects every orbit. For $n \in N$, we have the decomposition $T_{\epsilon(n)}s^{-1}(N) \cong A_{n} \oplus T_{n}N$, where $\epsilon(n)$ is the identity arrow at $n$. The map $dt_{\epsilon(n)}$ acts as the anchor on the first summand and the identity on the second. It is therefore surjective because $A$ is transverse to $N$. Hence $t|_{s^{-1}(N)}$ is a submersion along the identity bisection. Using left translation by local bisections, we can then show that $t|_{s^{-1}(N)}$ is a submersion everywhere.
\end{proof}

\section{Logarithmic connections: Global theory} \label{Globalsection}
Let $X$ be a complex manifold with a smooth closed hypersurface $D \subset X$ which is not necessarily connected, and let $G$ be a connected complex reductive group. In this section, we will study the $G$-representations of the twisted fundamental groupoid $\Pi(X,D)$. By Lie's second theorem, these representations are equivalent to principal $G$-bundles over $X$ equipped with flat connections which have a logarithmic singularity along $D$. The goal of this section is to establish a general functorial Riemann-Hilbert correspondence for logarithmic flat connections. 

\subsection{Logarithmic connections on line bundles}  \label{logconglinbun}
As a warm-up, we will start by studying the special case where $D$ is the zero section in the total space of a line bundle $p : L \to D$. This will serve as a local model for the general case which we will consider later. 

We start by describing in detail the structure of the twisted fundamental groupoid $\Pi(L, D)$. This groupoid has two orbits: the zero section $D$ and its complement $L^{\times}$. The restriction of the groupoid to the open orbit $L^{\times}$ is simply a fundamental groupoid 
\[
\Pi(L, D)|_{L^{\times}} = \Pi(L^{\times}).
\]
To determine the restriction of the groupoid to the orbit $D$, consider first the corresponding restriction of the Lie algebroid $T_{L}(- \log D)|_{D}$. This was observed in \cite{Gualtieri-Li-2012, tortella2016representations} to be isomorphic to the Atiyah algebroid $At(L^{\times})$, where $L^{\times}$ is viewed as a principal $\bb{C}^{*}$-bundle. To see this isomorphism, recall that the sections of $At(L^{\times})$ are defined to be the $\bb{C}^{*}$-invariant vector fields on $L$. Locally, these are spanned by $\bb{C}^{*}$-invariant lifts of vector fields on $D$, and the Euler vector field $z \partial_{z}$, where $z$ is a local linear coordinate on the fibers of $L$. These vector fields also define sections of $T_{L}(- \log D)$, and they are precisely the ones that survive when we restrict to the zero section. 

As a result of the isomorphism $T_{L}(-\log D)|_{D} \cong At(L^{\times})$, we have an isomorphism between the restriction $\Pi(L, D)|_{D}$ and the source simply connected integration of $At(L^{\times})$. This is given by the quotient $\Pi(L^{\times})/\bb{C}^{*}$, which is a groupoid over $D$. The composition of arrows is defined to be concatenation of paths, using the $\bb{C}^{*}$ action to match up the endpoints as required. 

There is a natural morphism of Lie algebroids $p_{!}: T_{L}(-\log D) \to At(L^{\times})$ which covers the projection map $p: L \to D$. Over $L^{\times}$ this map is given by the quotient projection $TL \to TL/\bb{C}^{*}$, and over $D$ it is given by the isomorphism of Lie algebroids described above. By Lie's second theorem, this map integrates to a morphism of Lie groupoids
\[
p : \Pi(L, D) \to \Pi(L^{\times})/\bb{C}^{*}.
\]
Combining $p$ with the source map $s$, we arrive at the following map 
\[
\phi = (p, s) : \Pi(L, D) \to \Pi(L^{\times})/\bb{C}^* \times_{D} L,
\]
which is an isomorphism. This defines a linear left action of $\Pi(L^{\times})/\bb{C}^{*}$ on $L$, which is given by the following formula 
\[
g.v = t( \phi^{-1}(g,v)),
\]
where $g \in \Pi(L^{\times})/\bb{C}^{*}$ and $v \in L$, such that $s(g) = p(v)$. As a result, $\phi$ defines an isomorphism between the twisted fundamental groupoid and the resulting action groupoid:
\[
\Pi(L, D) \cong \Pi(L^{\times})/\bb{C}^{*} \ltimes L.
\]

Given any point $y \in D$, the fibre $L_{y}$ is transverse to the Lie algebroid $T_{L}(- \log D)$ and intersects both orbits of $\Pi(L,D)$. Therefore, by Proposition \ref{criterion}, $\Pi(L,D)$ is Morita equivalent to the restriction 
\[
\Pi(L, D)|_{L_{y}} \cong A(L_{y}) \ltimes L_{y},
\]
where $A(L_{y}) = \Pi(L^{\times})/\bb{C}^{*} |_{y}$ is the \emph{isotropy group} of $\Pi(L^{\times})/\bb{C}^{*}$ at $y$, consisting of homotopy classes of paths in $L^{\times}$ which start and end on the given fibre $L_{y}^{\times}$. By Proposition \ref{repeq}, we therefore get an equivalence of representation categories 
\[
\Rep(\Pi(L,D), G) \cong \Rep(A(L_{y}) \ltimes L_{y}, G).
\]

We now take a closer look at this category of representations. In \cite[Lemma 4.1.3]{tortella2016representations} the isotropy group $A(L_{y})$ was shown to sit in the following short exact sequence 
\begin{equation} \label{Atisotropyseq}
0 \to \bb{Z} \to \pi(L^{\times}, v) \times \bb{C} \to A(L_{y}) \to 0, 
\end{equation}
where $\pi(L^{\times}, v)$ is the fundamental group of $L^{\times}$ based at a point $v \in L_{y}^{\times}$. Let $l$ denote the central element of $\pi(L^{\times}, v)$ which is represented by a counterclockwise loop contained in the fibre above $y$. The leftmost map in the sequence is defined by sending $n \in \bb{Z}$ to $(l^{-n}, 2 \pi i n) \in \pi(L^{\times}, v) \times \bb{C}$. This sequence can be upgraded to a short exact sequence of action groupoids 
\[
0 \to \bb{Z} \times L_{y} \to (\pi(L^{\times}, v) \times \bb{C}) \ltimes L_{y} \to A(L_{y}) \ltimes L_{y} \to 0.
\]
By pulling back representations along the maps in the sequence, we obtain functors between the corresponding representation categories which go in the opposite direction. The category $\Rep(A(L_{y}) \ltimes L_{y}, G)$ is thus embedded into $\Rep((\pi(L^{\times}, v) \times \bb{C}) \ltimes L_{y} , G)$ as the subcategory of representations which pullback to a trivial $\bb{Z} \times L_{y}$ representation. 

The representations of $(\pi(L^{\times}, v) \times \bb{C}) \ltimes L_{y}$ are easy to describe. They are given by a pair of representations of $\bb{C} \ltimes L_{y}$ and $\pi(L^{\times}, v)$ on the same principal bundle $P$ over $L_{y}$. Furthermore, because $\pi(L^{\times}, v)$ acts trivially on $L_{y}$ and commutes with $\bb{C}$, it acts on $P$ by automorphisms of the $\bb{C}\ltimes L_{y}$-representation. In terms of this description, the representations of $A(L_{y}) \ltimes L_{y}$ are simply those for which the actions of $2 \pi i \in \bb{C}$ and $l \in \pi(L^{\times}, v)$ agree. We summarise these observations in the following lemma. 

\begin{lemma} \label{furthersimplification}
The category of representations $\Rep(A(L_{y}) \ltimes L_{y}, G)$ is isomorphic to a category consisting of tuples $(P, \Phi, \gamma)$, where $(P, \Phi)$ is a representation of $\bb{C} \ltimes L_{y}$, and $\gamma : \pi(L^{\times}, v) \to \Aut(P, \Phi)$ is a homomorphism such that $\gamma(l) = \Phi(2 \pi i, z)$. 
\end{lemma}

We can now apply the functors from Theorem \ref{RH}  to give a description of the category of $\Pi(L,D)$ representations in terms of generalized monodromy data. To this end, we define a category $F((L,D), G)$, which is similar to the one defined in Section \ref{Definingthecat}. The objects of this category consist of tuples $(\gamma, K_{1}, \nu_{0}, K_{0}, A)$, where $K_{1}$ and $K_{0}$ are right $G$-torsors, $A \in \frak{aut}_{G}(K_{0})$ is an infinitesimal automorphism admitting the JC decomposition $A = S + N$, $\nu_{0} \subset \spc{Hom}_{G}(K_{0},K_{1})$ is a right $U_{\bb{N}}(S)$-torsor, and finally, $\gamma : \pi(L^{\times}, v) \to \text{St}(\nu) \subseteq Aut_{G}(K_{1})$ is a homomorphism. The compatibility condition for this tuple is that $\chi(\gamma(l)) = \exp(2 \pi i A)$, where $l$ is the central element represented by a counterclockwise loop in the fibre $L_{y}^{\times}$, and $\chi$ is the map defined by the sequence \ref{canonicalprojection}. The morphisms of this category are defined to be morphisms between the underlying $G$-torsors which preserve all of the structure. 

\begin{theorem} \label{RHforLB}
Let $L \to D$ be a holomorphic line bundle. Choose points $y \in Y$ and $v \in L_{y} \setminus \{ 0 \}$. Then there is an equivalence of categories
\[
\Rep(\Pi(L,D), G) \cong F((L,D), G). 
\]
\end{theorem}
\begin{proof}
The above discussion, Lemma \ref{furthersimplification}, and Theorem \ref{RH} combine to give an equivalence between $\Rep(\Pi(L,D), G)$ and a category consisting of tuples $(\gamma, M, K_{1}, \nu_{0}, K_{0}, A)$, where $(M, K_{1}, \nu_{0}, K_{0}, A)$ is an object of $F((\bb{C},0), G)$, and $\gamma = (\gamma_{1}, \gamma_{0}) : \pi(L^{\times}, v) \to \Aut(M, K_{1}, \nu_{0}, K_{0}, A)$ is a homomorphism such that $\gamma_{1}(l) = M$. The tuple $(\gamma_{1}, K_{1}, \nu_{0}, K_{0}, A)$ gives an object of $F((L,D), G)$. Therefore, it suffices to show that we can reconstruct $\gamma$ from $\gamma_{1}$ alone. 

Recall from Section \ref{Definingthecat} that an automorphism of the object $(M, K_{1}, \nu_{0}, K_{0}, A) \in F((\bb{C},0), G)$ consists of a pair of $G$-equivariant automorphisms $T_{i} : K_{i} \to K_{i}$, $i = 0, 1$, which preserve all of the structure. In particular, $T_{0} \in C_{\Aut_{G}(K_{0})}(A)$, and $T_{1} \nu_{0} T_{0}^{-1} = \nu_{0}$. As a result, $T_{1} \in St(\nu) \cap C_{\Aut_{G}(K_{1})}(M)$, and $\chi(T_{1}) = T_{0}$. This implies that the automorphism $(T_{1}, T_{0})$ is completely determined by $T_{1}$ alone. Therefore, the homomorphism $\gamma$ is fully determined by the component $\gamma_{1} : \pi(L^{\times}, v) \to St(\nu)$. 
\end{proof}

\subsection{Logarithmic connections on complex manifolds} 
In this section, we deal with the fully general case of a smooth hypersurface $D$ in a complex manifold $X$. The approach that we adopt for studying the representations of the twisted fundamental groupoid $\Pi(X,D)$ is to use a Morita equivalence to decompose the groupoid into more simple building blocks which can be analyzed using the methods of previous sections. These building blocks consist of two types: the fundamental group of the complement $X \setminus D$, and the twisted fundamental groupoid of the normal bundle $\Pi(N_{X,D}, D)$, which was studied in Section \ref{logconglinbun}.

\subsubsection{Construction of the Morita equivalence} \label{constructingME}
We construct a Morita equivalence relating $\Pi(X, D)$ to a groupoid which is defined over the normal bundle of $D$. The construction is inspired by Deligne's notion of \emph{tangential basepoints} \cite{deligne1989groupe}. The intuition is to consider paths connecting points in $X \setminus D$ to points in $D$, but we only identify paths if they are related by a homotopy that fixes the normal direction of the paths at their endpoints lying in $D$. Hence, we are in effect considering paths that have an endpoint lying in the normal bundle to $D$. In order to make this precise, we will construct a larger space $Z$ and a hypersurface $\tilde{D}$, which contains both $(X, D)$ and $(N_{X,D}, 0_{D})$, where $0_{D}$ is the zero section of the normal bundle $N_{X,D}$. The twisted fundamental groupoid $\Pi(Z, \tilde{D})$ of this larger pair of spaces will then allow us to consider paths going between $N_{X,D}$, and $X$. 

First, for each connected component $D_{i}$ of the hypersurface $D$, choose a tubular neighbourhood $U_{i} \subset X$. These must be chosen so that they are pairwise disjoint. Next, consider the deformation to the normal cone of $(U_{i}, D_{i})$ (see \cite{fulton2013intersection} for a general account of this construction). This is a family $p : \cal{D}_{i} \to \bb{C}$ which interpolates between $p^{-1}(t) = U_{i}$, for $t \neq 0$, and the normal bundle $p^{-1}(0) = N_{X, D_{i}}$. The space $\cal{D}_{i}$ can be constructed as follows 
\[
\cal{D}_{i} = \Bl(U_{i} \times \bb{C}, D_{i} \times \{ 0 \}) \setminus \overline{(U_{i} \setminus D_{i}) \times \{ 0 \} },
\]
where $\Bl(U_{i} \times \bb{C}, D_{i} \times \{ 0 \}) $ is the blow-up of $U_{i} \times \bb{C}$ at the codimension $2$ submanifold $D_{i} \times \{ 0 \}$, and $ \overline{(U_{i} \setminus D_{i}) \times \{ 0 \} }$ is the proper transform of the fibre $U_{i} \times \{0 \}$. The map $p$ is then given by the blow-down map followed by the projection to $\bb{C}$. 

The deformation space $\cal{D}_{i}$ contains a smooth hypersurface $\widetilde{D}_{i}$ which is isomorphic to $D_{i} \times \bb{C}$, and whose normal bundle is isomorphic to $N_{X,D_{i}} \times \bb{C}$. This hypersurface intersects the non-zero fibres at $D_{i}$, and the fibre over $0$ at the zero section. Hence, the pair $(\cal{D}_{i}, \widetilde{D}_{i})$ provides an interpolation between $(U_{i}, D_{i})$ and $(N_{X,D}, 0_{D})$. Furthermore, because of the existence of tubular neighbourhood embeddings in the smooth category, $(\cal{D}_{i}, \widetilde{D}_{i})$ is topologically equivalent to $(N_{X,D} \times \bb{C}, 0_{D} \times \bb{C})$. 

The disjoint union $\sqcup_{i} (\cal{D}_{i} \setminus p^{-1}(0))$ embeds holomorphically into $X \times \bb{C}^*$. We define 
\[
Z = (\sqcup_{i} \cal{D}_{i}) \cup (X \times D(1,r)),
\]
where $D(1, r)$ is a disc of radius $0 < r \ll 1$ centred at $1 \in \bb{C}$. There is a natural map $p : Z \to \bb{C}$ such that $p^{-1}(0) = \sqcup_{i} N_{X, D_{i}}$, and $p^{-1}(1) = X$. We also define the smooth hypersurface $\widetilde{D} = \sqcup_{i} \widetilde{D}_{i}$. Then $\widetilde{D} \cong D \times \bb{C}$, and it intersects $p^{-1}(1)$ along $D$, and $p^{-1}(0)$ along the disjoint union of the zero sections. 

Next consider the twisted fundamental groupoid $\Pi(Z, \widetilde{D})$, which is the source simply connected integration of $T_{Z}(- \log \widetilde{D})$. The orbits of this groupoid consist of $Z \setminus  \widetilde{D}$ and the hypersurfaces $\widetilde{D}_{i}$. The fibre $p^{-1}(1) = X$ intersects each orbit and is transverse to $T_{Z}(- \log \widetilde{D})$. Therefore, by Proposition \ref{criterion}, $\Pi(Z, \widetilde{D})$ is Morita equivalent to $\Pi(Z, \widetilde{D})|_{X}$. 
\begin{lemma}
The groupoid $\Pi(Z, \widetilde{D})|_{X}$ is canonically isomorphic to $\Pi(X, D)$. 
\end{lemma}
\begin{proof}
The groupoid $\Pi(Z, \widetilde{D})|_{X}$ is an integration of the Lie algebroid $T_{Z}(- \log \widetilde{D})|_{X} = T_{X}(-\log D)$. Therefore, the identity Lie algebroid morphism integrates by Lie's second theorem to a morphism of Lie groupoids $\psi : \Pi(X,D) \to \Pi(Z, \widetilde{D})|_{X}$. This map is an isomorphism when restricted to each component $D_{i}$, since the normal bundle of $\widetilde{D}$ is isomorphic to $N_{X,D} \times \bb{C}$. Hence, to show that $\psi$ is an isomorphism, it suffices to check that it induces an isomorphism between the fundamental groups of $X \setminus D$ and $Z \setminus \widetilde{D}$. This fact follows from a combination of the van Kampen theorem, and the fact that each pair $(\cal{D}_{i}, \widetilde{D}_{i})$ is homeomorphic to $(N_{X,D} \times \bb{C}, 0_{D} \times \bb{C})$. 
\end{proof}

Choose a basepoint $x_{0} \in X \setminus D$, and for each component $D_{i}$ choose a point $d_{i} \in D_{i}$. Let 
\[
V =\{(x_{0}, 1) \} \cup (\sqcup_{i} N_{X,D_{i}}|_{d_{i}}) \subset Z,
\]
and let $\cal{N} = \Pi(Z, \widetilde{D})|_{V}$. The submanifold $V$ satisfies the conditions of Proposition \ref{criterion}. Therefore $\cal{N}$ is a Lie subgroupoid which is Morita equivalent to $\Pi(Z, \widetilde{D})$. Composing the two Morita equivalences constructed in this section yields a Morita equivalence between $\Pi(X,D)$ and $\cal{N}$. This Morita equivalence is given by
\[
Q = t^{-1}(V) \cap s^{-1}(X),
\]
where $t$ and $s$ are the target and source maps of $\Pi(Z, \widetilde{D})$, respectively. The elements of $\cal{N}$ and $Q$ may be interpreted as paths in $X$ whose endpoints may lie in the normal bundle to $D$.

\subsubsection{Decomposing the groupoid of paths with tangential basepoints}
In this section, we give a van Kampen style decomposition of the groupoid $\cal{N}$ into simple building blocks. To this end, choose non-zero vectors $v_{i} \in N_{X,D_{i}}|_{d_{i}}$ for each component $D_{i}$, and let $\Pi(X \setminus D)_{\bar{v}}$ denote the restriction of $\cal{N}$ to the finite set of basepoints $\bar{v}$ consisting of all $v_{i}$ and $(x_{0}, 1)$. The groupoid $\Pi(X \setminus D)_{\bar{v}}$ can be interpreted as the fundamental groupoid of $X \setminus D$ relative to the basepoint $x_{0}$ and the tangential basepoints $v_{i}$. Recall from Section \ref{logconglinbun} the definition of the group $A(N_{X,D_{i}}|_{d_{i}}) = \Pi(N^{\times}_{X,D_{i}})/\mathbb{C}^*|_{d_{i}}$, which acts linearly on the fibre $N_{X,D_{i}}|_{d_{i}}$ of the normal bundle. 

\begin{proposition} \label{pushoutdecomp}
Let $X$ be a complex manifold with a smooth closed hypersurface $D \subset X$. Choose a point $x_{0} \in X \setminus D$, and for each connected component $D_{i}$ of $D$, choose a point $d_{i} \in D_{i}$, and a non-zero normal vector $v_{i} \in N_{X,D_{i}}|_{d_{i}}$.  Then $\Pi(X,D)$ is Morita equivalent to a Lie groupoid $\cal{N}$ whose base manifold is $V =\{x_{0} \} \sqcup (\sqcup_{i} N_{X,D_{i}}|_{d_{i}})$. Furthermore, $\cal{N}$ is the colimit in the category of holomorphic Lie groupoids of the following diagram 
\[
\xymatrix{ {\pi(N_{X,D_{i}}^{\times}, v_{i})} \ar[r]\ar[d] & { \Pi(X \setminus D)_{\bar{v}}} \\
{A(N_{X,D_{i}}|_{d_{i}}) \ltimes N_{X,D_{i}}|_{d_{i}} } }
\]
\end{proposition}
\begin{proof}
The first part of the proof follows from the discussion in Section \ref{constructingME}. It then remains for us to show that $\cal{N}$ is the colimit of the above diagram. For each component $D_{i}$, we have $N_{X,D_{i}} \subset Z$, and restricting the log tangent algebroid, we get the identification $T_{Z}(- \log \widetilde{D})|_{N_{X,D_{i}} } = T_{N_{X,D_{i}}}(- \log \ 0_{D_{i}})$. Hence, by Lie's second theorem we have a morphism $\Pi(N_{X,D_{i}}, D_{i}) \to \Pi(Z, \widetilde{D})$. Using the results of Section \ref{logconglinbun}, we see that this map can be restricted to the following morphism
\[
\phi_{i} : A(N_{X,D_{i}}|_{d_{i}}) \ltimes N_{X,D_{i}}|_{d_{i}}  \to \cal{N}.
\]
If we restrict $\phi_{i}$ to the origin $0_{d_{i}}$ we get an isomorphism. If we restrict instead to the point $v_{i}$, then we get the map $\pi(N_{X,D_{i}}^{\times}, v_{i}) \to \Pi(X \setminus D)_{\bar{v}}$, and this sits in the following commutative square
\[
\xymatrix{ {\pi(N_{X,D_{i}}^{\times}, v_{i})} \ar[r]\ar[d] & { \Pi(X \setminus D)_{\bar{v}}} \ar[d] \\
{A(N_{X,D_{i}}|_{d_{i}}) \ltimes N_{X,D_{i}}|_{d_{i}} } \ar[r] & {\cal{N}}.}
\]
Every arrow in $\cal{N}$ is a product of arrows from $\Pi(X \setminus D)_{\bar{v}}$ and arrows of the form $\phi_{i}(u)$, for $u \in A(N_{X,D_{i}}|_{d_{i}}) \ltimes N_{X,D_{i}}|_{d_{i}}$. Furthermore, the non-uniqueness of the factorization arises precisely from elements $\phi_{i}(\gamma)$, for $\gamma \in \pi(N_{X,D_{i}}^{\times}, v_{i})$. It follows from this that $\cal{N}$ is the desired colimit. 
\end{proof}
\begin{remark}
It is possible to give similar Morita equivalent models for the other integrations of $T_{X}(- \log D)$. For example, we can describe the maximal Hausdorff integration in the following way. Let $K_{i}$ be the kernel of $\pi(N_{X,D_{i}}^{\times}, v_{i}) \to \Pi(X \setminus D)_{\bar{v}}$. Then in the diagram from Proposition \ref{pushoutdecomp}, replace $\pi(N_{X,D_{i}}^{\times}, v_{i})$  by the quotient $\pi(N_{X,D_{i}}^{\times}, v_{i})/K_{i}$, and replace $A(N_{X,D_{i}}|_{d_{i}}) \ltimes N_{X,D_{i}}|_{d_{i}}$ by the quotient $\big(A(N_{X,D_{i}}|_{d_{i}})/K_{i} \big) \ltimes N_{X,D_{i}}|_{d_{i}}$. The corresponding colimit is Morita equivalent to the maximal Hausdorff integration of $T_{X}(-\log D)$. 

Alternatively, we can replace $\pi(N_{X,D_{i}}^{\times}, v_{i})$ by the trivial group, $A(N_{X,D_{i}}|_{d_{i}}) \ltimes N_{X,D_{i}}|_{d_{i}}$ by $\mathbb{C}^{*} \ltimes N_{X,D_{i}}|_{d_{i}}$, and $\Pi(X \setminus D)_{\bar{v}}$ by the groupoid over $\bar{v}$ which has a unique arrow between each pair of points. In this case, the colimit is Morita equivalent to the \emph{twisted pair groupoid} $Pair(X,D)$, which is the minimal integration of $T_{X}(- \log D)$ (see \cite[Section 3.3]{gualtieri2018stokes} for the definition). 
\end{remark}

\begin{remark}
A groupoid linearization is an isomorphism between a Lie groupoid, restricted to a neighborhood of one of its orbits, and an appropriately defined linear model. In the smooth category, Weinstein's trick \cite{weinstein2002linearization, crainic2013linearization} states that a groupoid is linearizable if and only if it is Morita equivalent to a linear model. This result fails in the holomorphic category where the existence of groupoid linearizations is obstructed by the lack of tubular neighborhood embeddings. 
Therefore, Morita equivalences between a groupoid and it's linear models provide the natural replacement for linearizations in the holomorphic category. Proposition \ref{pushoutdecomp} can be interpreted as a linearization result for $\Pi(X,D)$ in a neighbourhood of the orbits $D_{i}$. 
\end{remark}

\subsubsection{Riemann-Hilbert Correspondence} \label{RHsection}
In this section we prove a functorial Riemann-Hilbert classification theorem for logarithmic connections on $(X,D)$ in terms of generalized monodromy data. The Morita equivalence constructed in Section \ref{constructingME} induces, by Proposition \ref{repeq}, an equivalence of categories between $\Rep(\Pi(X,D), G)$ and $\Rep(\cal{N}, G)$. Proposition \ref{pushoutdecomp} implies that a representation of $\cal{N}$ decomposes into a representation of $\Pi(X \setminus D)_{\bar{v}}$ and a collection of representations of $A(N_{X,D_{i}}|_{d_{i}}) \ltimes N_{X,D_{i}}|_{d_{i}}$, subject to compatibility conditions. We can then apply the functors from Theorem \ref{RHforLB} to obtain a final classification result. 

We define a category $F((X,D), G)$ of generalized monodromy data, which we will use to classify logarithmic connections. In order to define this category, we first choose some auxiliary data. Let $\bar{d} = \{ d_{1}, ..., d_{n} \}$ be the collection of chosen points $d_{i} \in D_{i}$. For each point $d_{i}$, choose a path $p_{i} : [0,1] \to X$, such that $p_{i}(0) = d_{i}$, $p_{i}'(0) = v_{i}$, $p_{i}(1) = x_{0}$, and $p_{i}((0,1]) \subset X \setminus D$. Let $U_{i}$ be a tubular neighbourhood of $D_{i}$, and let $\lambda_{i}$ denote the small loop in $U_{i}$ which rotates once around $D_{i}$ in the counterclockwise direction. The path $p_{i}$ induces a homomorphism from $\pi(U_{i} \setminus D_{i})$ to $\pi(X \setminus D, x_{0})$. Denote the image of this homomorphism by $\pi_{i} \subset \pi(X \setminus D, x_{0})$, and denote the image of $\lambda_{i}$ by $l_{i} \in \pi_{i}$. 

The objects of $F((X,D), G)$ are defined to be tuples $(K, \Phi, A, \nu_{0})$, where $K \to \bar{d} \cup \{ x_{0} \}$ is a principal $G$-bundle (i.e. a collection of right $G$-torsors), $\Phi : \pi(X \setminus D, x_{0}) \to \Aut_{G}(K_{x_{0}})$ is a homomorphism, $A = (A_{1}, ..., A_{n})$ is a collection of infinitesimal symmetries $A_{i} \in \frak{aut}_{G}(K_{d_{i}})$ admitting JC decompositions $A_{i} = S_{i} + N_{i}$, and $\nu_{0} = (\nu_{0}^{(1)}, ..., \nu_{0}^{(n)})$ is a collection of torsors, where $\nu_{0}^{(i)} \subset \spc{Hom}_{G}(K_{d_{i}},K_{x_{0}})$ is a right $U_{\bb{N}}(S_{i})$-torsor. These data satisfy the compatibility conditions that $\Phi|_{\pi_{i}} : \pi_{i} \to \text{St}(\nu^{(i)})$ and $\chi_{i}(\Phi(l_{i})) = \exp(2 \pi i A_{i})$, for all $i$. Recall that the subgroups $\text{St}(\nu^{(i)}) \subseteq \Aut_{G}(K_{x_{0}})$ and the homormophisms $\chi_{i} : \text{St}(\nu^{(i)})  \to C_{\Aut_{G}(K_{d_{i}})}(S_{i})$ are defined in Section \ref{Definingthecat}. The morphisms in $F((X,D), G)$ are defined to be morphisms of $G$-bundles which intertwine all of the data. 

\begin{theorem} \label{FinalRH}
Let $X$ be a complex manifold with a smooth hypersurface $D \subset X$. Choose a basepoint $x_{0} \in X \setminus D$, and for each connected component $D_{i} \subseteq D$, choose a path $p_{i} : [0, 1] \to X$, such that $p_{i}(0) = d_{i} \in D_{i}$, $p_{i}'(0) = v_{i} \in N_{X, D_{i}}|_{d_{i}}$ is a non-zero normal vector, $p_{i}(1) = x_{0}$, and $p_{i}((0,1]) \subset X \setminus D$. Then there is an equivalence of categories 
\[
\Rep(T_{X}(-\log D), G) \cong F((X,D), G),
\]
where $\Rep(T_{X}(-\log D), G)$ is the category of flat connections on principal $G$-bundles over $X$ with logarithmic singularities along $D$, and $F((X,D), G)$ is the category of generalized monodromy data defined above. 
\end{theorem}
\begin{proof}
From the discussion in the present section, we have an equivalence \[ \Rep(T_{X}(-\log D), G) \cong \Rep(\cal{N}, G).\] It thus remains for us to prove the equivalence between $\Rep(\cal{N}, G)$ and ${F(\pi(X \setminus D,x_{0}), G)}$. By Proposition \ref{pushoutdecomp} and Theorem \ref{RHforLB}, we get an equivalence between $\Rep(\cal{N}, G)$ and a category $\mathcal{C}$ whose objects consist of the data of a principal bundle $K$ over $\bar{d} \cup \{x_{0}, v_{1}, ..., v_{n}\}$, a representation $\Psi$ of $\Pi(X \setminus X)_{\bar{v}}$ on this bundle, a collection of Lie algebra elements $A_{i} \in \frak{aut}_{G}(K_{d_{i}})$, and a collection of torsors $\mu_{0}^{(i)} \subset \spc{Hom}_{G}(K_{d_{i}},K_{v_{i}})$, all subject to appropriate compatibility conditions. 

For each path $p_{i}$, the path $(p_{i}(t), t) : [0,1] \to X \times \bb{C}$ can be lifted to a path $\tilde{p}_{i} \in \Pi(Z, \widetilde{D})$ going from $v_{i} \in N_{X,D_{i}}|_{d_{i}}$ to $(x_{0},1)$. Hence $\tilde{p}_{i} \in \Pi(X \setminus D)_{\bar{v}}$. Conjugating by this path sends the image of $\pi(N_{X,D_{i}}^{\times}, v_{i})$ in $\Pi(X \setminus D)_{\bar{v}}$ onto the subgroup $\pi_{i} \subseteq \pi(X \setminus D, x_{0})$, and the distinguished loop in the fibre $N_{X,D_{i}}^{\times}|_{d_{i}}$ to the element $l_{i} \in \pi_{i}$. Using these paths to identify $K_{v_{i}} \cong K_{x_{0}}$ we then get the desired equivalence between $\mathcal{C}$ and $F((X,D), G)$. 
\end{proof}

 \bibliographystyle{plain}

 \bibliography{bibliography.bib}

\end{document}